\documentclass[leqno, twoside]{amsart}
\usepackage{amsmath,amssymb,amscd,amsthm,amsxtra,amsfonts}
\usepackage{mathtools,mathrsfs,dsfont,xparse}
\usepackage{graphicx,epstopdf}
\usepackage{multirow}
\usepackage{cases}
\usepackage{enumerate}
\usepackage{xcolor}
\usepackage{tikz,pgfplots}
\usetikzlibrary{matrix}
\usepackage{mathtools}
\usepackage[margin=1 in]{geometry}
\mathtoolsset{showonlyrefs}
\usepackage[numbers,sort]{natbib}
\usepackage{mathrsfs}
\usepackage{enumitem}

\usepackage{booktabs} 

\newtheorem{theorem}{Theorem}[section]
\newtheorem{lemma}[theorem]{Lemma}
\newtheorem{proposition}[theorem]{Proposition}
\newtheorem{corollary}[theorem]{Corollary}

\newcommand{\R}{\mathbb{R}}
\newcommand{\T}{\mathbb{T}}

\newcommand{\beq}{\begin{equation}}
\newcommand{\eeq}{\end{equation}}
\newcommand{\beqq}{\begin{equation*}}
\newcommand{\eeqq}{\end{equation*}}

\theoremstyle{definition}
\newtheorem{definition}[theorem]{Definition}

\newtheorem{remark}[theorem]{Remark}

\numberwithin{equation}{section}

\usepackage[colorlinks=true, pdfstartview=FitV, linkcolor=blue, citecolor=blue, urlcolor=blue]{hyperref}

\DeclareMathOperator{\re}{Re}
\DeclareMathOperator{\im}{Im}

\allowdisplaybreaks[4]

\newcommand{\abs}[1]{\left|#1\right|}
\newcommand{\norm}[1]{\left\|#1\right\|}
\newcommand{\inner}[1]{\left \langle#1\right \rangle}
\newcommand{\sq}[1]{\left[#1\right]}

\numberwithin{equation}{section}

\begin{document}

\title[NLS with Ornstein-Uhlenbeck operator]{Nonlinear Schr\"odinger equation with Ornstein-Uhlenbeck operator}

\author{Xueying Yu, Haitian Yue and Zehua Zhao}

\address{Xueying Yu
\newline \indent Department of Mathematics, Oregon State University\indent 
\newline \indent  Kidder Hall 368
Corvallis, OR 97331}
\email{xueying.yu@oregonstate.edu}

\address{Haitian Yue
\newline \indent Institute of Mathematical Sciences, ShanghaiTech University\newline\indent
Pudong, Shanghai, China.}
\email{yuehaitian@shanghaitech.edu.cn}

\address{Zehua Zhao
\newline \indent Department of Mathematics and Statistics, Beijing Institute of Technology, Beijing, China.
\newline \indent Key Laboratory of Algebraic Lie Theory and Analysis of Ministry of Education, Beijing, China.}
\email{zzh@bit.edu.cn}



\begin{abstract}

In this work, we introduce and study nonlinear Schr\"odinger equations (NLS) with anisotropic dispersion, where the standard Laplacian acts on the Euclidean variable \(x \in \mathbb{R}^d\), and an Ornstein-Uhlenbeck ($\mathcal{OU}$) operator governs the confined direction \(\alpha \in \mathbb{R}\). We consider models with  two natural variants of $\mathcal{OU}$-induced confinement: \eqref{maineq2} based on the divergence form \(\nabla_\alpha \cdot (e^{-\frac{\alpha^2}{2}} \nabla_\alpha)\), and \eqref{maineq1} based on the non-divergence form \(\Delta_\alpha - \alpha \cdot \nabla_\alpha\). For both models, we establish the Strichartz estimates and Gaussian-weighted Morawetz estimates. In addition, for \eqref{maineq2}, we prove a virial-type finite-time blow-up result; for \eqref{maineq1}, we establish global well-posedness and small data scattering in the 2D quintic and 3D cubic cases. The primary motivation of this work is to capture waveguide-type dispersive behavior in a Euclidean setting. To the best of our knowledge, this is the first rigorous analysis of NLS with $\mathcal{OU}$ operators in both divergence and non-divergence forms.

\end{abstract}

\maketitle

\setcounter{tocdepth}{1}
\tableofcontents

\parindent = 10pt     
\parskip = 8pt

\section{Introduction}\label{1}

In this paper, we introduce a new class of nonlinear Schr\"odinger equations (NLS) incorporating Ornstein-Uhlenbeck ($\mathcal{OU}$) operators, and study their unified dispersive dynamics under drift confinement influenced by $\mathcal{OU}$ operators.

\subsection{Background and motivations}

As this is a new model, we will begin with an overview of its background. Referring to the seminal work \cite{OU} by Uhlenbeck and Ornstein, a model of Brownian motion with a drift component was introduced, which has subsequently become known as the Ornstein–Uhlenbeck process.
This process has been widely applied across various fields. In the context of optical fibers with refractive index fluctuations, the $\mathcal{OU}$ operator, $\nabla_\alpha \cdot (e^{-\frac{\alpha^2}{2}} \nabla_\alpha)$ is used to model disorder-induced localization. It is worth noting that this modification avoids the artificial periodicity of $\mathbb{T}^n$, which motivates our study of NLS with $\mathcal{OU}$ drift as a more flexible framework for non-compact confinement. Further details will be discussed in later sections.

The Ornstein–Uhlenbeck process can heuristically be viewed as a harmonic analysis model where the Gaussian measure replaces the usual Lebesgue measure. 
The Gaussian measure naturally serves as a substitute for the Lebesgue measure in infinite-dimensional settings, making the 
$\mathcal{OU}$ operator a fundamental tool in fields such as stochastic analysis and quantum field theory. This work aims to bridge the \textit{probabilistic framework} of $\mathcal{OU}$ processes with the analysis of dispersive PDEs.
 In $L^2$ with Gaussian measure, the $\mathcal{OU}$ operator can be viewed as a generalization of the Laplace operator, defined by the following form:
\begin{align}
\mathcal{OU} := \Delta_{\alpha} - \alpha \cdot \nabla_{\alpha} ,
\end{align}
where $\alpha$ is the spatial variable and it is self-adjoint.

\subsection{Divergence form vs non-divergence form}
As the readers may have observed, we used two different forms of $\mathcal{OU}$ operators. 
We will now take a moment to discuss the relations and differences between these two forms, $\nabla_\alpha \cdot (e^{-\frac{\alpha^2}{2}} \nabla_\alpha)$ and $\Delta_{\alpha} - \alpha \cdot \nabla_{\alpha}$. In fact, a direct computation shows
\begin{align}\label{eq div}
\nabla_\alpha \cdot (e^{-\frac{\alpha^2}{2}} \nabla_\alpha f) =  (\Delta_{\alpha} - \alpha \cdot \nabla_{\alpha} f ) e^{-\frac{\alpha^2}{2}} ,
\end{align}
which implies that they are identical up to a Gaussian weight $e^{-\frac{\alpha^2}{2}}$. They will, for example, exhibit identical spectral properties. In particular, we note that $\nabla_\alpha \cdot \left(e^{-\frac{\alpha^2}{2}} \nabla_\alpha \right)$ is expressed in divergence form.

We mention the divergence form since it is often favorable and convenient to utilize the divergence form of elliptic operators. 
For example, the Laplace operator $\Delta$ can be represented as the composition of the divergence and gradient operators:
\begin{align}
L u =  \Delta u =  \text{div} \nabla u .
\end{align}
Similarly, the variable coefficient Laplacian, which appears in heat diffusion problems, is often defined as:
\begin{align}
L u =  \text{div} (A(x) \nabla u) ,
\end{align}
where $A(x) = (a_{ij})$ is a symmetric and
positive definite matrix.

There are definite advantages to considering the two different representations of $\mathcal{OU}$ separately. As one recalls from PDE classes, the divergence form is most natural for energy methods, based upon integration by parts, and the nondivergence form is most appropriate for the maximum principle.
In the rest of the paper, we will constantly omit the subscript in our notation and refer to it simply as $\mathcal{OU}$ when there is no confusion.

\subsection{Connection to Gaussian measure}
As previously mentioned, let us establish the relationship between the $\mathcal{OU}$ operator to the standard Laplacian. 
To this end, we begin by examining the inner product in the Gaussian-weighted $L^2$-space. For simplicity, we will restrict our discussion to one dimension.
\begin{align}
\inner{\mathcal{OU} f, g}_{L^2 (e^{-\frac{\alpha^2}{2}})} & = \int_{\R} (\Delta_{\alpha} - \alpha \cdot \nabla_{\alpha} ) f \, g  \, e^{-\frac{\alpha^2}{2}} \, d\alpha \\
& = \int_{\R} \nabla_{\alpha} \cdot (\nabla_{\alpha} f e^{-\frac{\alpha^2}{2}} ) g \, d \alpha \\
& = \int_{\R} f \nabla_{\alpha} \cdot (\nabla_{\alpha} g e^{-\frac{\alpha^2}{2}} )  \, d \alpha = \inner{f, \mathcal{OU} g}_{L^2 (e^{-\frac{\alpha^2}{2}})} ,
\end{align}
where we used the identity \eqref{eq div}. 
Now we see that such $\mathcal{OU}$ operator is positive and symmetric and plays the role of the Laplacian on $L^2 (e^{-\frac{\alpha^2}{2}})$ (i.e., a weighted $L^2$-space). 
Similarly, we also have the same property for the divergence version
\begin{align}
\inner{\mathcal{OU}_{div} f ,g }_{L^2} = \inner{f, \mathcal{OU}_{div} g}_{L^2} .
\end{align}
Moreover, the Hermite polynomials are eigenvectors for the $\mathcal{OU}$ operator with discrete eigenvalues equal to natural numbers. This feature is very similar to the NLS models confined by harmonic potentials, which will be compared in later sections.

\subsection{NLS models confined with $\mathcal{OU}$ operators}

Introducing the $\mathcal{OU}$ operator to NLS would be of significant interest, as such a new model would exhibit certain ``compact-like” behaviors due to the discrete eigenvalues associated with the $\mathcal{OU}$ operator. Consequently, a partially confined model would mimic the waveguide-type behavior without necessarily being restricted to product spaces.
Here for waveguides, we refer to semi-periodic space $\mathbb{R}^m \times \mathbb{T}^n$ ($m,n\geq 1$). We will soon discuss the background of NLS on waveguides, which is an active research area in dispersive PDEs and has been intensively studied in recent decades.

Given the two different variants of the $\mathcal{OU}$ operator discussed earlier, we will list two natural confinement scenarios:
\begin{enumerate}
\item 
Introducing the \textit{divergence form} of $\mathcal{OU}$ operator $\mathcal{OU}_{div}$ to NLS (with power-type nonlinearity)
\begin{align}\label{maineq2}
i\partial_t u+\Delta_x u +\nabla_{\alpha} \cdot (\nabla_{\alpha} u\, e^{-\frac{\alpha^2}{2}} )  = |u|^p u . \tag{Model Div}
\end{align}
\item 
Introducing the non-divergence form of the $\mathcal{OU}$ operator directly to NLS, but with a slight weight modification of the nonlinearity\footnote{There are two motivations for the weight modification: the following transformation to connecting the two models and the nonlinear estimate (see Section \ref{WP}).}
\begin{align}\label{maineq1}
i\partial_t u+\Delta_x u+(\Delta_{\alpha}-\alpha \cdot \nabla_{\alpha})u=w(\alpha)|u|^p u, \tag{Model Non-Div}
\end{align}
where $w(\alpha) = e^{-\frac{p\alpha^2}{2}}$. 
\end{enumerate}
In fact, these are two different models, but directly connected via a multiplication of $e^{-\frac{\alpha^2}{2}}$ to \eqref{maineq1}, Identity \eqref{eq div} gives
\begin{align}
(i\partial_t+\Delta_x )u e^{-\frac{\alpha^2}{2}} + ( \Delta_{\alpha} - \alpha \cdot \nabla_{\alpha} )u e^{-\frac{\alpha^2}{2}}  = |e^{-\frac{\alpha^2}{2}}  u|^p (e^{-\frac{\alpha^2}{2}} u),\\
\implies (i\partial_t+\Delta_x )u e^{-\frac{\alpha^2}{2}} +  \nabla_{\alpha} \cdot ( e^{-\frac{\alpha^2}{2}} \nabla_{\alpha}  u)  = |e^{-\frac{\alpha^2}{2}}  u|^p (e^{-\frac{\alpha^2}{2}} u)
\end{align}

As we can see, \eqref{maineq2} and \eqref{maineq1} are tightly related but \textit{not} exactly equivalent. \eqref{maineq2} describes diffusion with Gaussian-weighted dispersion, while \eqref{maineq1} incorporates a drift term directly, analogous to a confining potential. We will study both of the two cases, i.e. {$\mathcal{OU}$ operator in divergence/non-divergence form} throughout this paper. In Section \ref{7}, we will remark further on this comparison.

For convenience, we consider the ``$\mathcal{OU}$ dimension" be one (i.e. $\alpha \in \mathbb{R}$ in \eqref{maineq2} and  \eqref{maineq1}). Higher dimensional models can be studied in a similar way.

When the ``$\mathcal{OU}$ dimension" is one and Euclidean dimension $d\geq 1$, \eqref{maineq2} reads,
\begin{equation}\label{maineqA}
    (i\partial_t+\Delta_x)u+\partial_{\alpha} \cdot ( e^{-\frac{\alpha^2}{2}} \partial_{\alpha} u )=|u|^p u,
\end{equation}
and \eqref{maineq1} reads
\begin{equation}\label{maineqB}
    (i\partial_t+\Delta_x)u+(\partial_{\alpha \alpha}-\alpha \cdot \partial_{\alpha})u=w(\alpha)|u|^p u,
\end{equation}
where $(x,\alpha)\in \mathbb{R}_x^d \times \mathbb{R}_{\alpha}$.

Since the nonlinear models \eqref{maineq2} and \eqref{maineq1} are newly introduced, our goal of this paper is to derive conservation laws, establish Strichartz estimates, develop well-posedness theory, and explore the existence of finite-time blow-up solutions in the focusing setting.

\subsection{The statement of the main results}
Now, let us state the results of this paper. As mentioned earlier, our results include various aspects: \textit{Strichartz estimates} (Section \ref{Stri}), \textit{well-posedness and small data scattering} (Section \ref{WP}), \textit{interaction Morawetz estimates} (Section \ref{Morawetz}), and \textit{finite-time blow-up results} (Section \ref{foc}). 
In particular, for the well-posedness and small data scattering, we focus on specific models (3D cubic model and 2D quintic model) to illustrate the ideas better; for the other three parts, we consider more general models.

We present a summary of our results along with an outline of the proof strategy.

For \eqref{maineq1}, we prove 
\begin{enumerate}[label=(ND-\Alph*)]
\item 
First, we establish Strichartz estimates for \eqref{maineq1}. We recall the equation with general nonlinearity $F$ as follows,
\begin{equation}\label{sec3: maineq1}
    (i\partial_t+\Delta_x+(\partial_{\alpha \alpha}-\alpha \partial_{\alpha}))u=F,\quad u(0,x,\alpha)=f(x,\alpha)\in H_x^r\mathcal{H}_{\alpha}^{\gamma}(\mathbb{R}_x^d \times \mathbb{R}_{\alpha}),
\end{equation}
where $r,\gamma \in \{0,1\}$. Here $\mathcal{L}_{\alpha}^2$ and $\mathcal{H}_{\alpha}^1$ indicate weighted $L^2$-based norms with respect to $\alpha$-direction, and we will give the notations in Subsection \ref{Notations}, together with the notion of Schr\"odinger admissible pairs. 

We show the following estimate.
\begin{theorem}[Strichartz Estimates for \eqref{maineq1}]\label{mainthm1}
Let $D$ denote $\partial_{x_j}$ ($j=1,2...,d$) and let $k=0,1$. Consider \eqref{sec3: maineq1} and Strichartz pairs $(p,q),(\tilde{p},\tilde{q})$ satisfying the Schr\"odinger admissible relations, then
\begin{equation}\label{rStri1}
    \|D^k u\|_{L^p_tL^q_x\mathcal{L}^2_{\alpha}}\lesssim \|D^k f\|_{L^2_{x}\mathcal{L}^2_{\alpha}}+\|D^k F\|_{L^{\tilde{p}^{'}}_tL^{\tilde{q}^{'}}_x\mathcal{L}^2_{\alpha}},
\end{equation}
and 
\begin{equation}\label{rStri2}
    \|D^k u\|_{L^p_tL^q_x\mathcal{H}^{1}_{\alpha}}\lesssim \|D^k f\|_{\mathcal{H}^{1}_{\alpha}}+\|D^k F\|_{L^{\tilde{p}^{'}}_tL^{\tilde{q}^{'}}_x\mathcal{H}^{1}_{\alpha}}.
\end{equation}   
\end{theorem}

To prove Theorem \ref{mainthm1}, we project the solution onto the eigenbasis of the $\mathcal{OU}$ operator (Hermite functions), reducing the problem to a family of Schr\"odinger equations on $\mathbb{R}^d$. The discrete spectrum of $\mathcal{OU}$ allows summing the classical Strichartz estimates uniformly over all modes. The one spatial-derivative (in $\mathcal{OU}$) case can be handled based on a ``lucky" observation.\footnote{This case can be handled due to a favorable structural identity.}\\

\item Based on Strichartz estimates, together with nonlinear estimates in the weighted setting (newly introduced in Section \ref{WP}), we establish the well-posedness and the small data scattering. 
More precisely, we prove nonlinear estimates in Lemma \ref{nonlinear2} and well-posedness and small data scattering for the cubic case of \eqref{maineq1} in Proposition \ref{prop: cubic2}.

We consider the 3D cubic model of \eqref{maineqB} as an example, i.e. 
\begin{equation}\label{maineq3dcubic}
    (i\partial_t+\Delta_x)u+(\partial_{\alpha \alpha}-\alpha \cdot \partial_{\alpha})u=e^{-\alpha^2}|u|^2 u,\quad u(0,x,\alpha)=u_0\in \mathcal{H}_{x,\alpha}^1(\mathbb{R}_x^2 \times \mathbb{R}_{\alpha}).
\end{equation}      
We have the following well posedness result,
\begin{theorem}[Well-posedness for \eqref{maineq1}]\label{mainthm2}
For any $E>0$, if $\|u_0\|_{\mathcal{H}^1_{x,\alpha}}\leq E$, there exists a unique global solution $u$ of \eqref{maineq3dcubic} satisfying
\begin{equation}
 \|u(t)\|_{L_t^{\infty}\mathcal{H}^1_{x,\alpha}} \lesssim \|u_0\|_{\mathcal{H}^1_{x,\alpha}}.  
\end{equation}    
\end{theorem}
See Subsection \ref{Notations} for the definition of $\mathcal{H}^1_{x,\alpha}$. This norm is a natural choice since it is compatible with the conservation law and the Strichartz estimate.

The main idea behind Theorem \ref{mainthm2} is based on Strichartz estimates (Theorem \ref{mainthm1}), a contraction mapping argument in the $L^p_t \mathcal{H}^1_\alpha$-type space is employed, leveraging the weighted Sobolev embedding $\mathcal{H}^1_\alpha \hookrightarrow \mathcal{L}^\infty_\alpha$ (Lemma \ref{delta2}) to control the nonlinearity (see Lemma \ref{nonlinear2}).

\begin{remark}
 Unlike the waveguide case where compactness arises from periodicity, our model achieves similar behavior through the spectral properties of the $\mathcal{OU}$ operator, allowing for non-product geometries.    
\end{remark}

\begin{remark}
Theorem \ref{mainthm2} concerns the defocusing case. We note that, for the local well-posedness theory, the result also holds for the focusing case (i.e., adding a negative sign to the nonlinearity in \eqref{maineq3dcubic}). However, more ingredients (such as ground state structures, variation analysis, and threshold assumptions) are needed to investigate the long-time dynamics for the focusing case. See Section \ref{foc} for blow-up results for the focusing case.
\end{remark}
\begin{remark}
In some respects, our $\mathcal{OU}$ models are technically more challenging than both the waveguide setting and the case with a partial harmonic potential. A key difficulty lies in the fact that, unlike the waveguide case, derivatives in the non-Euclidean (i.e., $\alpha$) direction do not commute with the linear Schr\"odinger flow under the $\mathcal{OU}$ operator. While a similar non-commutativity also arises in the harmonic potential case, the latter still enjoys local-in-time dispersive estimates, which are unavailable in our setting. Furthermore, our analysis requires developing new tools in Gaussian-weighted function spaces to accommodate the non-compact confinement structure. We refer to Section \ref{Stri}, Section \ref{WP} and Section \ref{7} for more details and discussions.    
\end{remark}
We also include the small data scattering result, and analogous results for the 2D quintic case in Proposition \ref{prop: quintic1} and Proposition \ref{prop: quintic2}. We also include discussions for \eqref{maineqB}. We refer to Section \ref{WP} for more details.\\

\item 
Interaction Morawetz estimates for \eqref{maineq1} in Proposition \ref{prop Morawetz2} and Corollary \ref{cor Morawetz2}.

\end{enumerate}

For \eqref{maineq2}, we prove
\begin{enumerate}[label=(D-\Alph*)]
\item 
We establish Strichartz estimates for \eqref{maineq2}. 
\begin{theorem}[Strichartz estimate for \eqref{maineq2}]\label{Strichartznew 2}
 Indicate by $D$ for $\partial_{x_j}$ ($j=1,2...,d$) and let $k=0,1$. Consider \eqref{sec3: maineq2} and Strichartz pairs $(p,q),(\tilde{p},\tilde{q})$ satisfying the NLS admissible relations, then we have,
\begin{equation}\label{rStri2 2}    \|D^k u\|_{L^p_tL^q_x L^2_{\alpha}}\lesssim \|D^k f\|_{L^2_xL^2_{\alpha}}+\|D^k F\|_{L^{\tilde{p}^{'}}_tL^{\tilde{q}^{'}}_x L^2_{\alpha}}.
\end{equation}   
\end{theorem}
Different from Theorem \ref{mainthm1}, the proof follows from standard Keel-Tao's machinery \cite{keeltaoendpoint} and a dispersive-type estimate. See Section \ref{Stri} for more details. We also discuss the well-posedness results in Section \ref{WP}.

\item Interaction Morawetz estimates for \eqref{maineq2} in Proposition \ref{prop Morawetz} and Corollary \ref{cor Morawetz}.

\item 
Next, we consider the \textit{focusing} NLS with $\mathcal{OU}$ operator in \eqref{maineq2} as follows,
\begin{align}\label{eq focusing1}
i\partial_t u+\Delta_x u + \nabla_{\alpha} \cdot (\nabla_{\alpha} u\, e^{-\frac{\alpha^2}{2}} ) = - \abs{u}^p u,  \qquad (t,x , \alpha) \in [0,T) \times \R^d \times \R.
\end{align}
We have,
\begin{theorem}[Finite-time blow-up for \eqref{maineq2}]\label{mainthm3}
Consider \eqref{eq focusing1} with $d\geq 1, p\geq \frac{4}{d}$, the initial data $u_0 \in H_{x, \alpha}^1$, $\|xu_0(x,\alpha)\|_{L^2_{x,\alpha}} < \infty$, and assume the energy $E(u_0) <0$, then for the maximal lifespan $[0,T)$, we have $T < \infty$.
\end{theorem}
For the definition of the energy $E$, see Proposition \ref{prop: Conserv}.

To see blow-up solutions, we need a partial Virial argument that adapts the classical approach by selecting a weight $|x|^2$ independent of $\alpha$, exploiting the Gaussian measure's decay to handle the $\mathcal{OU}$ term.

\begin{remark}
As a comparison with the classical blow-up theory for NLS, a ``partial Virial identity" is investigated and utilized in the proof of Theorem \ref{mainthm3}.   
\end{remark}
\end{enumerate}

To the best of the authors' knowledge, the current paper is the \textbf{first} work to study the dynamics of NLS with $\mathcal{OU}$ operators in both the divergence vs. non-divergence forms, like \eqref{maineq2} and \eqref{maineq1}. The Strichartz estimates and the Gaussian-weighted interaction Morawetz estimates in this setting are also new and have their own interests. There are numerous problems that can be further studied regarding this topic, which will be discussed in Section \ref{7}.

We hope the techniques developed here can be applied to other stochastic-inspired confinement mechanisms, including fractional $\mathcal{OU}$ operators or degenerate diffusions, possibly extending the interplay between stochastic analysis and nonlinear wave dynamics.

\subsection{Comparisons with models on waveguides and models with compact-like behaviors}
The spectral properties of the $\mathcal{OU}$ operator make it behave like a compact dimension, hence we will first compare it with the classical waveguide setting.

The waveguide manifold $\R^m \times \T^{d-m}$ is a product of the Euclidean space with tori (rational or irrational), and is of particular interest in nonlinear optics of telecommunications. In modern backbone networks, data signals are almost exclusively transmitted by optical carriers in fibers (a special case of a waveguide).  Applications like the internet demand an increase in the available bandwidth in the network and a reduction of costs for the transmission of data. The nonlinear Schr\"odinger type of model is of particular importance in the description of nonlinear effects in optical fibers.

In physics, optical waveguides confine and guide light by constraining waves to travel along a certain desired path. One interesting feature of studying the behavior of solutions on the waveguide manifold is that it mixed inherits properties from those on classical Euclidean spaces and tori, which captures well the physics behind it. The waveguide manifold’s product structure leads to NLS solutions inheriting mixed properties from Euclidean and tori settings.

The Euclidean case has been extensively studied, and its theory is well established, particularly for defocusing nonlinearities \cite{CW, CKSTTIM, Dod3d, Dod1d, Dod4d, BourHL, Dod2d, B2, KM06, B99, KVnotes, RV, V, Taobook}. Moreover, we refer to \cite{HTT1, IPT3, KV16, Yue} for a few works on tori. Due to the nature of such product spaces, we see that NLS posed on the waveguide manifold mixed inheriting properties from those on Euclidean spaces and tori. The techniques from Euclidean theory (e.g., Strichartz estimates) and compact manifolds (e.g., spectral decomposition) are synthesized to handle the waveguide problems. We refer to \cite{R2T,CGZ,CZZ,HP,HTT1,HTT2,IPT3,IPRT3, Z1,Z2,ZZ}  for some NLS results in the waveguide setting. See also \cite{forcella2020large,luo2025well,yu2024global,lyu2025strichartz} for some other equations in the waveguide setting as examples.
Based on the aforementioned existing results and theories, we expect that the solution of NLS exists globally and scatters in the range $\frac{4}{m}  \leq p \leq \frac{4}{d-2} $ with data in the energy space $H^1$ on the waveguide manifold $\R^m \times \T^{d-m}$\footnote{It is easy to see that $d-m \leq 2$ is required if the scattering behavior is expected. In other words, there cannot be ``too many" tori since that would cause insufficient dispersion. It is possible to enlarge the scattering range if one adds additional assumptions for the initial data, such as the higher regularity condition or in a weighted space.}. We refer to \cite{takaoka20012d,Barron,MR4782142} and the references therein for Strichartz-type estimates for Schr\"odinger equations in the context of waveguide manifolds. Finally, we refer to the introductions of some recent works \cite{kwak2024critical,kwak2024global,deng2024bilinear} on the topic of NLS on tori/waveguides.

It is interesting to generalize this idea in the sense that one would mimic the waveguide-type behavior by considering models not necessarily posed on product spaces, but admitting certain ``compact-like" behavior in some spatial direction. As an illustration, a well-known model,  NLS confined by the harmonic potential (to convey the idea effectively, we state the model in two spatial dimensions), 
\begin{align}\label{eq HarPot}
i\partial_t u + \Delta u - (\beta^2 \abs{x}^2 +   \abs{\alpha}^2 )u = \pm \abs{u}^p u , \quad \text{where } (x,\alpha) \in \R  \times \R 
\end{align}
arises in diverse physical phenomena, including Bose-Einstein condensates in a laboratory trap \cite{JP, PSS} and the envelope dynamics of a general dispersive wave in a weakly nonlinear medium. The parameter $\beta =0$ or $1$, with $\beta =1$ corresponding to the quadratic potential case, while $\beta =0$ corresponds to the partial harmonic oscillator. 
Due to the spectrum of the harmonic oscillator
\begin{align*}
\mathcal{H} : = \Delta_\alpha - \abs{\alpha}^2 
\end{align*}
being purely discrete (the spectrum of it is given by the set of integers $\{2n + 1, n \in \mathbb{N} \}$), one anticipates that such model \eqref{eq HarPot} with $\omega =0$ behaviors like NLS posed on $\R \times \T$ (where $y$ is the ``compact-like" direction), which has been confirmed in the study on global behaviors of such model \cite{PTV, CGGLS, Jao1, Jao2, antonelli2015scattering}. The large data scattering for \eqref{eq HarPot} with $p=4$ (the quintic case), $\beta=1$ and $\alpha=0$ is recently proven in \cite{ma2024scattering}.

As we mentioned earlier, another operator enjoys a similar ``compact-like" behavior, that is, the $\mathcal{OU}$ operator $
\mathcal{OU}  = \Delta_\alpha - \alpha \cdot \nabla_\alpha
$ (also known as the Hermite operator, since the equation with the operator of the form $D^2- x D$   is used in Hermite's paper \cite{Hermite}).
Tracing back to the classical paper \cite{OU}  by Uhlenbeck and Ornstein, a model of Brownian motion with a drift was presented, which soon came to be called the Ornstein–Uhlenbeck process. 
The aforementioned harmonic oscillator operator is also sometimes called the Hermite operator and is unitarily equivalent to the $\mathcal{OU}$ operator (if the former is considered acting in $L^2$ with respect to Lebesgue measure and the latter acting in $L^2$ with respect to the Gaussian measure). \footnote{The Gaussian weight $e^{-\frac{\alpha^2}{2}}$ induces localization, analogous to a compact domain. This contrast illustrates how $\mathcal{OU}$ drift simulates compact-like localization in an anisotropic fashion.}

Although the $\mathcal{OU}$ operator is well known and has been very popular in physics, and probability,  this model has not been widely introduced in this dispersive community. Our motivation is to introduce this different (nonlinear) model which captures the same phenomena as the harmonic partial confinement but as a drift, not a potential. As expected, the NLS with $\mathcal{OU}$ operator imitates a product manifold space, where the first spatial variable behaves normally in a Euclidean space while the other one acts in a ``compact-like" space. This configuration captures the essential features of waveguide systems while allowing more general geometries, as evidenced by the spectral analogy.

Note that the $\mathcal{OU}$ operator has a number of features similar to those of the Laplace operator, and its semigroup theory and properties are very well developed (see for example \cite{Bog}). As hinted in \cite{Barron, TV1, TV2}, one would consider projecting the equation along the eigenfunctions of the Laplace–Beltrami operator on $\T$ when working on product spaces,  for example, $\R \times \T$, and this results in obtaining a sequence of Schr\"odinger equations on $\R$. For each equation, the corresponding classical Strichartz estimates are perfectly applicable, and then summing up the sequence in an appropriate sense would yield the Strichartz on the product space. 
For the Morawetz type of estimates, one could handle in a similar manner as in \cite{TV2, yu2024global, SYYZ}, that is, treating the compact direction $\T$, roughly speaking, as a free dimension, and only choosing weight functions and computing interaction between the Euclidean variables.

\subsection{Organization of the rest of this paper}
The rest of this paper is organized as follows: in Section \ref{Cons}, we derive the conservation laws; in Section \ref{Stri}, we establish Strichartz estimates for our models (adapting the idea of \cite{TV1,TV2}) for \eqref{maineq1} and Keel-Tao's machinery \cite{keeltaoendpoint} for \eqref{maineq2}); in Section \ref{WP}, we establish the well-posedness theory (based on Strichartz estimates established in Section \ref{Stri} and weighted-type estimates), and the small data scattering; in Section \ref{Morawetz}, we establish Gaussian-weighted interaction-Morawetz type estimates; in Section \ref{foc}, we prove finite-time blow-up results in the focusing setting for \eqref{maineq2}; in Section \ref{7}, we make a summary and give further remarks.

\subsection*{Acknowledgment}
We highly appreciate Prof. Yannick Sire and Prof. Chenjie Fan for many beneficial discussions and suggestions. X. Yu is partially supported by NSF DMS-2306429. H. Yue was supported by the NSF grant of China (No.12301300). Z. Zhao was supported by the NSF grant of China (No. 12271032, 12426205) and the Beijing Institute of Technology Research Fund Program for Young Scholars.

\section{Preliminarily}\label{sec Pre}
In this section, we present notations and recall the classical Strichartz estimates.

\subsection{Notations}\label{Notations}
We write $A \lesssim B$ if there exists a constant $C$ such that $A\leq CB$. We use $A \simeq B$ when $A \lesssim B \lesssim A $. Particularly, we write $A \lesssim_u B$ to express that $A\leq C(u)B$ for some constant $C(u)$ depending on $u$. We use the usual Lebesgue spaces $L^{p}$ and  and Sobolev spaces $H^{s}$.

We regularly refer to the composed spacetime norms
\begin{equation}
    \|u\|_{L^p_tL^q_x(I_t \times \mathbb{R}^m)}=\left(\int_{I_t}\left(\int_{\mathbb{R}^m} |u(t,x)|^q \, dx \right)^{\frac{p}{q}}  \, dt\right)^{\frac{1}{p}}.
\end{equation}

Similarly, we can define the composition of three $L^p$-type norms like $L^p_tL^q_xL^2_{\alpha}$.

Now, we define the weighted function spaces, which are very useful for our models due to their structures. We denote the weighted norm $\mathcal{L}^2_{\alpha}$ as follows,
\begin{equation}\label{eq L^2}
\|u(\alpha)\|^2_{\mathcal{L}^2_{\alpha}}= \int_{ \R} \abs{u (\alpha)}^2 e^{-\frac{\alpha^2}{2}} \,  d\alpha.
\end{equation}
Similarly, $\dot{\mathcal{H}}^1_{\alpha}$ and $\mathcal{H}^1_{\alpha}$ are defined as follows,
\begin{equation}  \|u(\alpha)\|^2_{\dot{\mathcal{H}}^1_{\alpha}}= \int_{ \R} \abs{\partial_{\alpha} u (\alpha)}^2 e^{-\frac{\alpha^2}{2}} \,  d\alpha,
\end{equation}
\begin{equation}  \|u(\alpha)\|^2_{\mathcal{H}^1_{\alpha}}= \int_{ \R} \abs{\partial_{\alpha} u}^2 (\alpha) e^{-\frac{\alpha^2}{2}} \,  d\alpha+\int_{ \R} \abs{ u (\alpha)}^2 e^{-\frac{\alpha^2}{2}} \,  d\alpha.
\end{equation}
As for the composed case, for convenience, we refer to norms: $\mathcal{H}_{x,\alpha}^1(\mathbb{R}_x^d \times \mathbb{R}_{\alpha})$ and $L^p_x\mathcal{L}^2_{\alpha}(\mathbb{R}_x^d \times \mathbb{R}_{\alpha})$ to be
\begin{equation}  \|u(\alpha)\|^2_{\mathcal{H}^1_{x,\alpha}(\mathbb{R}_x^d \times \mathbb{R}_{\alpha})}= \int_{ \R^d} \int_{ \R} \abs{\partial_{\alpha} u (x,\alpha)}^2 e^{-\frac{\alpha^2}{2}} \,  d\alpha dx+\int_{ \R^d} \int_{ \R} \abs{\nabla_{x} u (x,\alpha)}^2e^{-\frac{\alpha^2}{2}}  \,  d\alpha dx+\int_{ \R^d} \int_{ \R} \abs{ u (x,\alpha)}^2e^{-\frac{\alpha^2}{2}}  \,  d\alpha dx,
\end{equation}
 and 
\begin{equation}  \|u(x,\alpha)\|^2_{L^p_x\mathcal{L}^2_{\alpha}}=\big( \int_{\mathbb{R}^d} \big(\int_{ \R} \abs{ u (x,\alpha)}^2 e^{-\frac{\alpha^2}{2}} \,  d\alpha \big)^{\frac{p}{2}} dx\big)^{\frac{1}{p}},
\end{equation}
respectively. Such weighted norms are very useful for our models (see conservation laws in Section \ref{Cons}, Strichartz estimates in Section \ref{Stri} and interaction Morawetz estimates in Section \ref{Morawetz} for more details).
\subsection{Strichartz estimates}
Then we recall the Strichartz pair and the classical Strichartz estimate as follows (see \cite{Taobook, keeltaoendpoint})\footnote{The Euclidean Strichartz estimate will be used as a blackbox to prove Strichartz estimates for our models.}.
\begin{definition} \label{def:Strichartz}
Let $q ,r \in [2, +\infty]$ and $(q,r,d)\neq (2,\infty,2)$. We say $(q,r)$ is an admissible Strichartz pair in $\R^d$ if
\begin{equation*}
\frac{2}{q} + \frac{d}{r} = \frac{d}{2}\;.
\end{equation*}
We say $(q,r)$ is admissible with $s$ ($s>0$) regularity if
\begin{align}
\frac{2}{q} + \frac{d}{r} = \frac{d}{2}-s.
\end{align}
\end{definition}
\begin{lemma}[Strichartz estimate] \label{Strichartz}
Let $(q,r)$ be an admissible Strichartz pair in $\R^d$. Then we have the bound
\begin{equation}
    \|e^{it\Delta} f\|_{L^q_t L^r_x(\mathbb{R}\times \mathbb{R}^d)} \lesssim \|f\|_{L^2(\mathbb{R}^d)}.
\end{equation}
Also, for any two Strichartz pairs $(q_1, r_1)$ and $(q_2,r_2)$, we have
\begin{equation}
    \bigg\| \int_0^t e^{i(t-s)\Delta}F(s)ds \bigg\|_{L^{q_1}_tL^{p_1}_x(\mathbb{R}\times \mathbb{R}^d)} \lesssim \|F\|_{L^{q_{2}^{'}}_{t} L^{p_{2}^{'}}_x(\mathbb{R}\times \mathbb{R}^d)}\;,
\end{equation}
where $q_2'$ and $r_2'$ are conjugates of $q_2$ and $r_2$. 
\end{lemma}

\section{Conservation Laws}\label{Cons}
In this section, we present the conservation laws for NLS with partial $\mathcal{OU}$ drifts. The conservation laws for both models follow from direct calculations, analogous to the classical NLS case (see \cite{Taobook}). We will cover both of the two models: \eqref{maineq2} and \eqref{maineq1}. We have,
\begin{proposition}\label{prop: Conserv}
The equation \eqref{maineq2} enjoys the following conservation laws,
\begin{itemize}
\item
Mass
\begin{align}
M(u)(t) : = \int_{\R^d \times \R} \abs{u}^2 (t,x , \alpha) \, dx d\alpha  \equiv M(u)(0) ;
\end{align}

\item 
Energy
\begin{align}
E(u)(t) : =  \int_{\R^d \times \R}  \frac{1}{2} \abs{\nabla_x u}^2 +  \frac{1}{2} \abs{\nabla_{\alpha} u}^2 e^{-\frac{\alpha^2}{2}} +  \frac{1}{p+2} \abs{u}^{p+2}  \, dx d\alpha \equiv E(u)(0) .
\end{align}

\end{itemize}
Moreover, the equation \eqref{maineq1} enjoys the following conservation laws,
\begin{itemize}
\item
Mass
\begin{align}
M(u)(t) : = \int_{\R^d \times \R} \abs{u}^2 (t,x , \alpha) e^{-\frac{\alpha^2}{2}} \, dx d\alpha  \equiv M(u)(0) ;
\end{align}
\item 
Energy
\begin{align}
E(u)(t) : =  \int_{\R^d \times \R}  [\frac{1}{2} \abs{\nabla_x u}^2 +  \frac{1}{2} \abs{\nabla_{\alpha} u}^2  +  \frac{1}{p+2} \abs{u}^{p+2} \omega(\alpha) ] e^{-\frac{\alpha^2}{2}}  \, dx d\alpha \equiv E(u)(0) .
\end{align}
\end{itemize}
\end{proposition}
\begin{remark}
As we can see in Proposition \ref{prop: Conserv}, for the energy of \eqref{maineq2}, there is a weight for the $|\nabla_{\alpha}u|^2$ term. If there is no such weight, the energy of \eqref{maineq2} is comparable with the standard $H^1$-norm. Moreover, for \eqref{maineq1}, the norm $\mathcal{H}^1_{x,\alpha}$ defined in Subsection \ref{Notations} is compatible with the whole energy (the sum of energy and mass). 
\end{remark}

\begin{proof}[Proof of Proposition \ref{prop: Conserv}]
We start with the conservation of mass. 
Taking the time derivative of the mass, we write by \eqref{maineq2} and process the integration by parts,
\begin{align}
\partial_t \int_{\R^d \times \R} \abs{u}^2 (t,x, \alpha)  \, dx d\alpha & = 2 \re \int_{\R^d \times \R} u_t \overline{u} (t,x,\alpha)  \, dx d\alpha \\
& = 2 \re \int_{\R^d \times \R}  -i [\abs{u}^p u  - \Delta_x u - \nabla_{\alpha} \cdot (\nabla_{\alpha} u\, e^{-\frac{\alpha^2}{2}} ) ] \overline{u} \, dx d\alpha \\
& = 2 \im \int_{\R^d \times \R} [\abs{u}^p u  - \Delta_x u - \nabla_{\alpha} \cdot (\nabla_{\alpha} u\, e^{-\frac{\alpha^2}{2}} ) ] \overline{u}  \, dx d\alpha \\
& = 2 \im \int_{\R^d \times \R} [- \Delta_x u - \nabla_{\alpha} \cdot (\nabla_{\alpha} u\, e^{-\frac{\alpha^2}{2}} )] \overline{u}  \, dx d\alpha \\
& = 2 \im \int_{\R^d \times \R} \abs{\nabla_x u}^2  \, dx d\alpha + 2 \im \int_{\R^d \times \R} \nabla_{\alpha} \overline{u} \cdot ( \nabla_{\alpha} u\, e^{-\frac{\alpha^2}{2}})  \, dx d\alpha \\
& = 2 \im \int_{\R^d \times \R}  \abs{\nabla_{\alpha} u}^2 e^{-\frac{\alpha^2}{2}}   \, dx d\alpha =0.
\end{align}

Similarly, for \eqref{maineq1}, using integration by parts, we have,
\begin{align}
\partial_t \int_{\R^d \times \R} \abs{u}^2 (t,x, \alpha) e^{-\frac{\alpha^2}{2}} \, dx d\alpha & = 2 \re \int_{\R^d \times \R} u_t \overline{u} (t,x,\alpha) e^{-\frac{\alpha^2}{2}}  \, dx d\alpha \\
& = 2 \re \int_{\R^d \times \R}  [-i (\abs{u}^p u \omega(\alpha) - \Delta_x u - \Delta_{\alpha} u + \alpha \cdot \nabla_{\alpha} u) \overline{u}] e^{-\frac{\alpha^2}{2}}  \, dx d\alpha \\
& = 2 \im \int_{\R^d \times \R} [(\abs{u}^p u \omega(\alpha) - \Delta_x u - \Delta_{\alpha} u + \alpha \cdot \nabla_{\alpha} u) \overline{u}] e^{-\frac{\alpha^2}{2}}  \, dx d\alpha \\
& = 2 \im \int_{\R^d \times \R} [(- \Delta_x u - \Delta_{\alpha} u + \alpha \cdot \nabla_{\alpha} u) \overline{u}] e^{-\frac{\alpha^2}{2}}  \, dx d\alpha \\
& = 2 \im \int_{\R^d \times \R} \abs{\nabla_x u}^2 e^{-\frac{\alpha^2}{2}}  \, dx d\alpha + 2 \im \int [ - \Delta_{\alpha} u + \alpha \cdot \nabla_{\alpha} u) \overline{u} ] e^{-\frac{\alpha^2}{2}} \, dx d\alpha \\ 
& = 2 \im \int_{\R^d \times \R} [-\nabla_{\alpha} \cdot ( \nabla_{\alpha} u \,e^{-\frac{\alpha^2}{2}}) \overline{u}] \, dx d\alpha \\
& = 2 \im \int_{\R^d \times \R}  \abs{\nabla_{\alpha} u}^2 e^{-\frac{\alpha^2}{2}}   \, dx d\alpha =0,
\end{align}
where we note that we used the following fact in the second last line above
\begin{align}\label{eq OU}
( \Delta_{\alpha} - \alpha \cdot \nabla_{\alpha} )u e^{-\frac{\alpha^2}{2}} = \nabla_{\alpha} \cdot (\nabla_{\alpha} u\, e^{-\frac{\alpha^2}{2}} )  .
\end{align}

Next, we turn to the computation of the conservation of energy. We will calculate the time derivative of the energy. Using integration by parts and \eqref{maineq2}, we have
\begin{align}
\frac{d}{dt} E(u)(t) & = \partial_t \int_{\R^d \times \R}  \frac{1}{2} \abs{\nabla_x u}^2 +  \frac{1}{2} \abs{\nabla_{\alpha} u}^2 e^{-\frac{\alpha^2}{2}} +  \frac{1}{p+2} \abs{u}^{p+2}  \, dx d\alpha \\
& = \re \int_{\R^d \times \R}  \nabla_{x} u \cdot \partial_t \nabla_{x} \overline{u} + \nabla_{\alpha} u \cdot \partial_t \nabla_{\alpha} \overline{u}\, e^{-\frac{\alpha^2}{2}} + \abs{u}^p u \partial_t \overline{u} \, dx d\alpha \\
& = \re \int_{\R^d \times \R} - \Delta_x u  \partial_t \overline{u}  - \nabla_{\alpha} \cdot (\nabla_{\alpha} u\, e^{-\frac{\alpha^2}{2}}) \partial_t \overline{u} + \abs{u}^p u \partial_t \overline{u} \, dx d\alpha \\
& = \re \int_{\R^d \times \R} \partial_t \overline{u} [- \Delta_x u  - \nabla_{\alpha} \cdot (\nabla_{\alpha} u\, e^{-\frac{\alpha^2}{2}})  + \abs{u}^p u  ]  \, dx d\alpha \\
& = \re \int_{\R^d \times \R} \partial_t \overline{u} [i \partial_t u] \, dx d\alpha = 0 .
\end{align}

The case for the model \eqref{maineq1} is also similar. We have
\begin{align}
\frac{d}{dt} E(u)(t) & = \partial_t \int_{\R^d \times \R}  [\frac{1}{2} \abs{\nabla_x u}^2 +  \frac{1}{2} \abs{\nabla_{\alpha} u}^2  +  \frac{1}{p+2} \abs{u}^{p+2} \omega(\alpha) ] e^{-\frac{\alpha^2}{2}} \, dx d\alpha \\
& = \re \int_{\R^d \times \R}  \nabla_{x} u \cdot \partial_t \nabla_{x} \overline{u} \, e^{-\frac{\alpha^2}{2}} + \nabla_{\alpha} u \cdot \partial_t \nabla_{\alpha} \overline{u} \,e^{-\frac{\alpha^2}{2}} + \abs{u}^p u \partial_t \overline{u} \omega(\alpha) e^{-\frac{\alpha^2}{2}} \, dx d\alpha \\
& = \re \int_{\R^d \times \R} - \Delta_x u  \partial_t \overline{u} \, e^{-\frac{\alpha^2}{2}} - \nabla_{\alpha} \cdot (\nabla_{\alpha} u\, e^{-\frac{\alpha^2}{2}}) \partial_t \overline{u} + \abs{u}^p u \partial_t \overline{u} \omega(\alpha) e^{-\frac{\alpha^2}{2}} \, dx d\alpha \\
& = \re \int_{\R^d \times \R} - \Delta_x u  \partial_t \overline{u} \, e^{-\frac{\alpha^2}{2}} - \Delta_{\alpha} u\, e^{-\frac{\alpha^2}{2}} \partial_t \overline{u}  + \nabla_{\alpha} u \cdot \alpha e^{-\frac{\alpha^2}{2}}   \partial_t \overline{u} + \abs{u}^p u \partial_t \overline{u} \omega(\alpha) e^{-\frac{\alpha^2}{2}} \, dx d\alpha \\
& = \re \int_{\R^d \times \R} \partial_t \overline{u} [- \Delta_x u  - \Delta_{\alpha} u + \alpha \cdot \nabla_{\alpha} u + \abs{u}^p u  \omega(\alpha)] e^{-\frac{\alpha^2}{2}}   \, dx d\alpha \\
& = \re \int_{\R^d \times \R} \partial_t \overline{u} [i \partial_t u] \, dx d\alpha = 0 .
\end{align}

The deductions of the conservation laws in Proposition \ref{prop: Conserv} for both of the two models \eqref{maineq2} and \eqref{maineq1} are now complete.
\end{proof}

\section{Strichartz Estimates}\label{Stri}
In this section, we discuss Strichartz estimates for \eqref{maineq2} and \eqref{maineq1}. In particular, we will give the proof for Theorem \ref{mainthm1}.
\subsection{Strichartz Estimates for \eqref{maineq1}}\label{Stri B}
In this subsection, we discuss the Strichartz estimates for \eqref{maineq1} by giving the proof for Theorem \ref{mainthm1}. These estimates are elementary for establishing the well-posedness theory. Unlike periodic/compact manifolds where spectral theory is standard, our method employs Hermite expansions of the $\mathcal{OU}$ operator and constructs time-space estimates via decomposition, which may be of independent interest. Compared to the waveguide case, another obstacle is that the spatial derivative in $\alpha$ does not commute with the linear equation\footnote{The situation for \eqref{maineq2} is even worse. For Theorem \ref{mainthm1} (\eqref{maineq1}), one derivative (in $\alpha$) case can be luckily handled, while the fractional derivative case is more complicated.}. We refer to Lemma \ref{Strichartz} for classical Strichartz estimates in the Euclidean setting, which serve as a blackbox in our proof. 

For convenience, we recall the equation \eqref{sec3: maineq1} as follows,
\begin{equation}
    (i\partial_t+\Delta_x+(\partial_{\alpha \alpha}-\alpha \partial_{\alpha}))u=F,\quad u(0,x,\alpha)=f(x,\alpha)\in H_x^r\mathcal{H}_{\alpha}^{\gamma}(\mathbb{R}_x^d \times \mathbb{R}_{\alpha}),
\end{equation}
where $r,\gamma \in \{0,1\}$. We give the proof for Theorem \ref{mainthm1} as follows.

\begin{proof}[Proof of Theorem \ref{mainthm1}]
We consider the non-derivative first (i.e. \eqref{rStri1} with $k=0$). We decompose the three functions (nonlinear solution, nonlinearity, initial data),
\begin{equation}
    u(t,x,\alpha), F(t,x,\alpha) \textmd{ and } f(x,\alpha)
\end{equation}
with respect to the orthonormal basis of $\mathcal{L}^2_{\alpha}(\mathbb{R})$ given by the associated eigenfunctions (Hermite polynomials) respectively, such that
\begin{align}
u(t,x,\alpha)&=\sum_j u_j(t,x) \phi_j(\alpha),\\
F(t,x,\alpha)&=\sum_j F_j(t,x) \phi_j(\alpha),\\
f(x,\alpha)&=\sum_j f_j(x) \phi_j(\alpha),
\end{align}
and notice that $u_j(t,x)$, $F_j(t,x)$ and $f_j(t,x)$ are related by the new Cauchy
problem as follows,
\begin{equation}\label{stric2}
 (i\partial_t+\Delta_x-\lambda_j)u_j=F_j\quad (t,x)\in \mathbb{R}_t \times  \mathbb{R}_x^d  
\end{equation}
with
\begin{equation}
   u_j(0,x)=f_j(x).
\end{equation}
Applying the usual Strichartz estimate on $\mathbb{R}^d$ (Lemma \ref{Strichartz}), we have
\begin{equation}
\|u_j(t,x)\|_{L^p_tL^q_x} \lesssim \|f_j\|_{L_x^2}+\|F_j\|_{L^{\tilde{p}^{'}}_tL^{\tilde{q}^{'}}_x} 
\end{equation}
and hence summing in $j$ the squares we obtain
\begin{equation}
\|u_j(t,x)\|_{l^2_j L^p_tL^q_x} \lesssim \|f\|_{L_x^2\mathcal{L}^2_{\alpha}}+\|F_j\|_{l^2_jL^{\tilde{p}^{'}}_tL^{\tilde{q}^{'}}_x}.     
\end{equation}
Then by the Minkowski, we obtain
\begin{equation}
\|u_j(t,x)\|_{ L^p_tL^q_xl^2_j} \lesssim \|f\|_{L_x^2\mathcal{L}^2_{\alpha}}+\|F_j\|_{L^{\tilde{p}^{'}}_tL^{\tilde{q}^{'}}_xl^2_j},
\end{equation}
with the Plancharel identity gives
\begin{equation}
\|u\|_{L^p_tL^q_x\mathcal{L}^2_{\alpha}}\lesssim \|f\|_{L^2_{x}\mathcal{L}^2_{\alpha}}+\|F\|_{L^{\tilde{p}^{'}}_tL^{\tilde{q}^{'}}_x\mathcal{L}^2_{\alpha}}.
\end{equation}
The derivative case\footnote{We need to emphasize that, the non-integer Sobolev regularity case is non-trivial, which is very different from the waveguide case \cite{TV1,TV2}.} follows in the same way. Taking the derivative for both sides in \eqref{stric2} first, we have
\begin{equation}
    (i\partial_t+\Delta_x+(\partial_{\alpha \alpha}-\alpha \partial_{\alpha})-1)(\partial_{\alpha} u)=\partial_{\alpha}F.   
\end{equation}
Then one can repeat the analysis for the non-derivative case one more time to obtain the desired estimate.

We now complete the proof of Theorem \ref{mainthm1}.
\end{proof}

\begin{remark}
 We note that the Strichartz estimates have similar tastes with the Strichartz estimates for the waveguide case (see \cite{TV1,TV2}). The main idea is to decompose the $\mathcal{OU}$ direction and apply the usual Strichartz estimates for each component.\footnote{ This decomposition method is useful for many other models, such as NLS with a partial harmonic potential (see \cite{antonelli2015scattering}) and NLS on waveguide manifolds (see \cite{TV1,TV2}). For the waveguide case, the tori component can be replaced by a general compact manifold. One may compare this method with another method mentioned in \cite{antonelli2015scattering} (see Proposition 3.1.): fixing the partial harmonic potential direction (or tori direction for the waveguide case) by $L^2$-norm, then using the dispersive estimate and applying Keel-Tao's scheme \cite{keeltaoendpoint} to obtain the desired Strichartz estimate.} Moreover, it is natural and trivial to extend it to the ``arbitrary $\mathcal{OU}$ dimensional case". The $\mathcal{OU}$ dimension does not differ. 
\end{remark}

\begin{remark}
The norms in Theorem \ref{mainthm1} concern \textit{Gaussian-weighted norms} regarding the $\mathcal{OU}$ direction, which is another difference from the waveguide case or the harmonic potential case. Based on Strichartz estimates, one expects to establish  well-posedness theory for \eqref{maineq1} with the proper choice of $p$. See Section \ref{WP} for more details.    
\end{remark}
\begin{remark}
While our estimates cover integer-order derivatives in the $\alpha$ ($\mathcal{OU}$)-direction, obtaining optimal Strichartz estimates for fractional derivatives remains an interesting open problem, which may reveal finer dispersive regularity, particularly in comparison with the waveguide case.
\end{remark}

\subsection{Strichartz Estimates for \eqref{maineq2}}

We now turn to \eqref{maineq2}. For convenience, we use $f$ for the initial data and $F$ for the general nonlinearity. We consider the $\mathcal{OU}$ direction is one dimensional and the generalization to the higher dimensional case is natural. We will use a different method and we will explain why the previous method is not adaptable for this case.

We consider \eqref{maineq2}. We recall the equation with general nonlinearity $F$ as follows.
\begin{align}\label{sec3: maineq2}
i\partial_t u+\Delta_x u +\partial_{\alpha} \cdot (\partial_{\alpha} u\, e^{-\frac{\alpha^2}{2}} ) 
 = F ,\quad u(0,x,\alpha)=f(x,\alpha)\in H^{\gamma}_xL^2_{\alpha}(\mathbb{R}_x^d \times \mathbb{R}_{\alpha}),
\end{align}
where $\gamma=0,1$. We have,
\begin{lemma}[Strichartz estimate]\label{Strichartznew 2}
 Indicate by $D$ for $\partial_{x_j}$ ($j=1,2...,d$) and let $k=0,1$. Consider \eqref{sec3: maineq2} and Strichartz pairs $(p,q),(\tilde{p},\tilde{q})$ satisfying the NLS admissible relations, then we have,
\begin{equation}\label{rStri2 2}    \|D^k u\|_{L^p_tL^q_x L^2_{\alpha}}\lesssim \|D^k f\|_{L^2_xL^2_{\alpha}}+\|D^k F\|_{L^{\tilde{p}^{'}}_tL^{\tilde{q}^{'}}_x L^2_{\alpha}}.
\end{equation}   
\end{lemma}

\begin{proof}[Proof of Lemma \ref{Strichartznew 2}]
For convenience, we write the operator $Pu=\partial_{\alpha} \cdot (\partial_{\alpha} u\, e^{-\frac{\alpha^2}{2}} )$. First we note
\begin{equation}
e^{it(\Delta_x+P)}=e^{itP}e^{it\Delta_x}=e^{it\Delta_x}e^{itP}.    
\end{equation}
Using this observation, the assumption that is bounded on for positive time, and the standard properties of the Schr\"odinger group in the Euclidean setting, 
\begin{equation}
  \|e^{it(\Delta_x+P)} u_0\|_{L^{2}_{x,\alpha}} \lesssim \|u_0\|_{L^{2}_{x,\alpha}} ,
\end{equation}
and
\begin{equation}
  \|e^{it(\Delta_x+P)} u_0\|_{L^{\infty}_xL^2_{\alpha}} \lesssim t^{\frac{d}{2}} \|u_0\|_{L^1_xL^2_{\alpha}}.    
\end{equation}
Invoking Keel-Tao's machinery (\cite{keeltaoendpoint}, Theorem 10.1), with
\begin{equation}
B_0=H=L^{2}_{x,\alpha},B_1=L^1_xL^2_{\alpha},   
\end{equation}
we then have the estimates in Lemma \ref{Strichartznew 2}.
\end{proof}
In the end of this section, we now remark on ``\textit{why we cannot use the same `Tzvetkov-Visciglia' method as \eqref{maineq1} for \eqref{maineq2}}'', even for the non-derivative first. 

Similarly, we decompose the three functions,
\begin{equation}
    u(t,x,\alpha), F(t,x,\alpha) \textmd{ and } f(x,\alpha),
\end{equation}
with respect to the orthonormal basis of $\mathcal{L}^2_{\alpha}(\mathbb{R})$ given by the associated eigenfunctions $\{\phi_j(\alpha)\}$ (which are Hermite polynomials) such that
\begin{align}
u(t,x,\alpha)&=\sum_j u_j(t,x) \phi_j(\alpha),\\
F(t,x,\alpha)&=\sum_j F_j(t,x) \phi_j(\alpha),\\
f(x,\alpha)&=\sum_j f_j(x) \phi_j(\alpha).
\end{align}
 In view of \eqref{sec3: maineq2} and Identity \eqref{eq div}, we see that $\phi_j(\alpha)$ satisfies 
\begin{equation}
\nabla_{\alpha} \cdot (\nabla_{\alpha} \phi_j(\alpha) e^{-\frac{\alpha^2}{2}} )=(\Delta_{\alpha} - \alpha \cdot \nabla_{\alpha} )\phi_j(\alpha) e^{-\frac{\alpha^2}{2}}=\lambda_j  \phi_j(\alpha) e^{-\frac{\alpha^2}{2}}.   
\end{equation}

Here $\phi_j(\alpha) e^{-\frac{\alpha^2}{2}}$ forms the orthonormal basis of $L^2_{\alpha}(\mathbb{R})$. We can now see \textbf{the problem} is: unlike the proof for \eqref{maineq1} (Theorem \ref{mainthm1}), for the $\mathcal{OU}$-part, an extra weight $e^{-\frac{\alpha^2}{2}}$ will appear. Therefore, this approach breaks down.

Finally we note that, if \eqref{maineq2} is modified a little bit (adding the weight to other terms) as follows,
\begin{align}\label{newnew}
e^{-\frac{\alpha^2}{2}}(i\partial_t u+\Delta_x) u +\partial_{\alpha} \cdot (\partial_{\alpha} u\, e^{-\frac{\alpha^2}{2}} ) 
 = e^{-\frac{\alpha^2}{2}}F.
\end{align}
We can use the same method as \eqref{maineq1} to establish Strichartz estimate, which is analogous to Theorem \ref{mainthm1}. In fact, \eqref{newnew} is essentially identical to \eqref{maineq1} in view of the transformation between \eqref{maineq2} and \eqref{maineq1}, as stated in Introduction.

These Strichartz estimates reveal the dispersive behavior induced by the $\mathcal{OU}$ drift and allow us to adapt standard contraction arguments to this mixed geometry setting. We now proceed to establish the well-posedness results based on these estimates.

\section{Well-posedness and Small Data Scattering}\label{WP}
Building upon the Strichartz estimates established in Section \ref{Stri}, we discuss the well-posedness and small data scattering in this section. We will first discuss well-posedness theory for \eqref{maineq1} by giving the proof for Theorem \ref{mainthm2}. Before presenting the proof, we remark on the \textbf{necessity} of the weight $w(\alpha)$ in \eqref{maineq1}. We explain why well-posedness theory cannot be handled in the associated energy space $\mathcal{H}^1_{\alpha}$ (corresponding to the conservation law) \textbf{without} the extra weight assumption from the aspect of estimating the nonlinearity. 

In view of Proposition \ref{prop: Conserv} and Theorem \ref{mainthm1}, it is natural to consider $\mathcal{H}^1_{\alpha}$-type weighted space. In general, one needs the following estimate to deal with the nonlinearity: there exists $C=C(p)>0$,
\begin{equation}
    \|w(\alpha)|u|^pu\|_{\dot{\mathcal{H}}_{\alpha}^1} \leq C \|u\|^{p+1}_{\dot{\mathcal{H}}_{\alpha}^1}.
\end{equation}
However, without the weight (letting $w(\alpha)\equiv 1$), this estimate is not expected to hold for a general function $u$. Consider $p=2$ for convenience. For example, taking $u \sim e^{\frac{\alpha^2}{8}}$, one can check that the right hand side is finite while the left hand side blows up.\footnote{Another explanation for the weight is due to the transformation between \eqref{maineq2} and \eqref{maineq1}, as stated in Introduction.}

With the same reason, to study \eqref{maineq2}, one needs to consider the usual $H^1_{\alpha}$-space instead of $\mathcal{H}^1_{\alpha}$-space in view of the nonlinearity (there is no extra weight).

\subsection{Basic settings: Conservation Laws, Strichartz estimates, nonlinear estimates}\label{Basic}
Conservation laws and Strichartz estimates are fundamental for further studying \eqref{maineq1}. They are included in Section \ref{Cons} and Section \ref{Stri}. They are elementary for obtaining the well-posedness theory.

Based on the Strichartz estimates in Section \ref{Stri}, the proof of the well-posedness theory for \eqref{maineq1} is expected to follow from the waveguide case \cite{R2T,TV2} with suitable modifications once we have good \textit{nonlinear estimates}. 

The main difference/challenge is that we will use weighted norms regarding the $\mathcal{OU}$-direction for our case, which is according to the Strichartz estimate (see Section \ref{Stri}). Thus some lemmas for the weighted case are needed.

For example, unlike \cite{TV2}, our contraction mapping requires Gaussian-weighted norms due to the \(\mathcal{OU}\) drift, necessitating Lemma \ref{delta2} for weighted Sobolev embeddings as follows,
\begin{lemma}[Gaussian-weighted Sobolev inequality]\label{delta2}
There exists an universal constant $C>0$ such that,
\begin{equation} 
\|u\|_{\mathcal{L}_{\alpha}^{\infty}(\mathbb{R})} \leq C \|u\|_{\mathcal{H}_{\alpha}^1(\mathbb{R})}.
\end{equation}
\end{lemma}
\begin{proof}[Proof of Lemma \ref{delta2}]
 We note that, according to the definition, the square of the right-hand side is:
 \begin{equation}\label{estSob}
     \int_{ \R} \abs{\partial_{\alpha} u (\alpha)}^2  e^{-\frac{\alpha^2}{2}} \,  d\alpha + \int_{ \R} \abs{ u(\alpha)}^2  e^{-\frac{\alpha^2}{2}} \,  d\alpha , 
 \end{equation}
 and the left-hand side is:
\begin{equation}
    \|u(\alpha) e^{-\frac{\alpha^2}{2}}\|_{L^{\infty}_{\alpha}}.
\end{equation}
Applying the Sobolev in 1D and Product rule, we have
\begin{equation}
 \|u e^{-\frac{\alpha^2}{2}}\|_{L^{\infty}_{\alpha}} \lesssim \|\partial_{\alpha} (u e^{-\frac{\alpha^2}{2}}))\|_{L^2_{\alpha}}\lesssim \|  (-\alpha)e^{-\frac{\alpha^2}{2}}u)\|_{L^2_{\alpha}}+\|\partial_{\alpha}u (e^{-\frac{\alpha^2}{2}}))\|_{L^2_{\alpha}},   
\end{equation}
which is bounded by the right-hand side of \eqref{estSob}.

This completes the proof of Lemma \ref{delta2}.
\end{proof}
To handle the nonlinearity, we need nonlinear estimates with respect to weighted norms. We consider the general case as follows,
\begin{lemma}[Nonlinear estimate for the weighted case]\label{nonlinear2}
For every integer $p>0$ there exists $C = C(p)>0$ such that
\begin{equation}
    \|e^{-\frac{p \alpha^2}{2}}|u|^pu\|_{\mathcal{H}_{\alpha}^1} \leq C \|u\|^{p+1}_{\mathcal{H}_{\alpha}^1}.
\end{equation}
\end{lemma}

\begin{proof}[Proof of Lemma \ref{nonlinear2}]
The $\mathcal{L}^2$ case is easier in view of the following estimate and Lemma \ref{delta2},
\begin{equation}
\|e^{-\frac{p \alpha^2}{2}}|u|^pu\|_{\mathcal{L}^2_{\alpha}} \lesssim  \|u\|_{\mathcal{L}^2_{\alpha}} \|u\|^p_{\mathcal{L}^{\infty}_{\alpha}}. 
\end{equation}
We now focus on the $\dot{\mathcal{H}}^1$ case.

Using the definitions of the weighted norms, Product rule, Lemma \ref{delta2}, the Sobolev and the H\"older, the square of the left hand side is bounded by: 
\begin{align}
&  \int_{ \R} \abs{\partial_{\alpha} (e^{-\frac{p \alpha^2}{2}}|u|^pu)}^2  e^{-\frac{\alpha^2}{2}} \,  d\alpha \\
& \lesssim  \int_{ \R} \abs{(p+1)(\partial_{\alpha} u)(e^{-\frac{p \alpha^2}{2}}|u|^p)}^2  e^{-\frac{\alpha^2}{2}} \,  d\alpha  \\
&+ \int_{ \R} \abs{(-p\alpha) (e^{-\frac{p \alpha^2}{2}}|u|^pu)}^2  e^{-\frac{\alpha^2}{2}} \,  d\alpha  \\
& \lesssim \|u\|^2_{\mathcal{H}_{\alpha}^1}\|u\|^{2p}_{\mathcal{L}_{\alpha}^{\infty}}+\|u\|^2_{\mathcal{H}_{\alpha}^1}\|\alpha^2 e^{-p\alpha^2} |u|^{2p}\|_{L^{\infty}_{\alpha}}\\
& \lesssim  \|u\|^{2p+2}_{\mathcal{H}_{\alpha}^1}+\|u\|^2_{\mathcal{H}_{\alpha}^1}\|u\|^2_{\mathcal{H}_{\alpha}^1} \|u\|^{2p-2}_{\mathcal{L}_{\alpha}^{\infty}}\\
& \lesssim  \|u\|^{2p+2}_{\mathcal{H}_{\alpha}^1}.
\end{align} 
The proof is finished.
\end{proof}
The above lemma plays a similar role as Lemma 4.1 in \cite{TV2}. It indicates the distribution of the derivative on the nonlinearity (with Gaussian weights)\footnote{We note that, for the above two lemmas, it is possible to consider more general cases (fractional derivatives or non-integer $p$), which would be more technical. We refer to \cite{V} for fractional calculus. Since we concern the cubic case and the quintic case in energy space, the nonlinear estimates would be much easier.}. 

For \eqref{maineq2}, according to Strichartz estimates in Section \ref{Stri} and its nonlinear structure, it suffices to use the following usual nonlinear estimate since no-weighted norms are used.
\begin{lemma}[Lemma 4.1 in \cite{TV2}]\label{delta}
For every $0<s<1$, $p>0$ there exists $C = C(p, s)>0$ such that
\begin{equation}
    \||u|^pu\|_{\dot{H}_{\alpha}^s} \leq C \|u\|_{\dot{H}_{\alpha}^s}\|u\|^p_{L_{\alpha}^{\infty}}.
\end{equation}
\end{lemma}
\begin{remark}\label{nonlinear}
Applying the Sobolev in one dimensional space, one can directly obtain, for $s> \frac{1}{2}$
\begin{equation}
    \||u|^pu\|_{\dot{H}_{\alpha}^s} \leq C \|u\|^{p+1}_{\dot{H}_{\alpha}^s},
\end{equation}
which is helpful to analyze the nonlinearity.\footnote{We note that if a higher dimensional case is considered, the $H^{1+}$-regularity is required due to the Sobolev.} \end{remark}
\subsection{Well-posedness Theory}
With the key estimates established, we are ready to turn to the well-posedness.
\begin{proof}[Proof of Theorem \ref{mainthm2}]

It suffices to show the local well-posedness. Global well-posedness follows from the conservation laws immediately. Once we have the Strichartz estimates (Theorem \ref{mainthm1}) and the nonlinear estimate (Lemma \ref{nonlinear2}) established, the proof is standard so we give a sketch as follows.   

Define a sequence $u_n$ as follows
\begin{equation}
    u_0=H(t) u_0,
\end{equation}
\begin{equation}
    u_{n+1}=H(t)u_n+i\int_0^t H(t-s)|u_n(s)|^2u_n(s)ds.
\end{equation}
Consider finite interval $I$. Then we have the following nonlinear estimates,
\begin{equation}
   \|u_{n+1}\|_{L_{t,x}^4(I)\mathcal{H}^1_{\alpha}}\leq C(I)\|u_0\|_{L_x^2 \mathcal{H}^1_{\alpha}}+C(I) \|u_{n}\|^3_{L_{t,x}^{4}(I)\mathcal{H}^1_{\alpha}}
\end{equation}
and
\begin{equation}
   \|u_{n+1}-u_n\|_{L_{t,x}^4(I)\mathcal{H}^1_{\alpha}}\leq C(I) \|u_{n+1}-u_n\|_{L_{t,x}^4(I)\mathcal{H}^1_{\alpha}}\times (\|u_{n+1}\|^2_{L_{t,x}^4(I)\mathcal{H}^1_{\alpha}}+\|u_{n}\|^2_{L_{t,x}^4(I)\mathcal{H}^1_{\alpha}}). 
\end{equation}
 Therefore $u_n$ converges to a local solution in the time interval $[-1,1]$ if one takes $I$ small enough.
\end{proof}
We also have,
\begin{proposition}[Small data scattering]\label{prop: cubic2}
There exists $\delta > 0$ such that if $u_0 \in \mathcal{H}^1_{x,\alpha}$ and
\begin{align}
\|u_0\|_{L^2_x \mathcal{H}^1_{\alpha}} \leq \delta,
\end{align}
then \eqref{maineq3dcubic} has a unique global solution $u(t,x,\alpha) \in C^0_t L^2_x \mathcal{H}^1_{\alpha} \cap L^4_t L^4_x \mathcal{H}^1_{\alpha}$, and there exist $u_\pm \in L^2_x \mathcal{H}^1_{\alpha}$ such that
\begin{align}
\|u(t,x,\alpha) - e^{it\Delta_{x,\alpha}} u_\pm(x,\alpha)\|_{L^2_x \mathcal{H}^1_{\alpha}} \to 0, \quad \text{as } t \to \pm \infty.
\end{align}
\end{proposition}
\begin{proof}
The proof is also standard once we have all necessary estimates. See Theorem 2.5 in \cite{R2T} and the references therein.    
\end{proof}
\begin{remark}
Following the idea of \cite{fan2024dispersive,fan2024decaying}, nonlinear dispersive estimates are expected to be further obtained.    
\end{remark}
Now we turn to the 2D quintic model as follows. 
\begin{equation}\label{maineq2dquintic}
    (i\partial_t+\partial_{xx})u+(\partial_{\alpha \alpha}-\alpha \cdot \partial_{\alpha})u=e^{-2\alpha^2}|u|^4 u,\quad u(0,x,\alpha)=u_0\in \mathcal{H}_{x,\alpha}^1(\mathbb{R}_x \times \mathbb{R}_{\alpha}).
\end{equation}  
We then formulate analogous results.
\begin{proposition}\label{prop: quintic1}
For any $E>0$, if $\|u_0\|_{\mathcal{H}^1_{x,\alpha}}\leq E$, there exists a unique global solution $u$ of \eqref{maineq2dquintic} satisfying
\begin{equation}
 \|u(t)\|_{L_t^{\infty}\mathcal{H}^1_{x,\alpha}} \lesssim \|u_0\|_{\mathcal{H}^1_{x,\alpha}}.  
\end{equation}
\end{proposition}
\begin{proposition}\label{prop: quintic2}
There exists $\delta > 0$ such that if $u_0 \in \mathcal{H}^1_{x,\alpha}$ and
\begin{align}
\|u_0\|_{L^2_x \mathcal{H}^1_{\alpha}} \leq \delta,
\end{align}
then \eqref{maineq2dquintic} has a unique global solution $u(t,x,\alpha) \in C^0_t L^2_x \mathcal{H}^1_{\alpha} \cap L^6_t L^6_x \mathcal{H}^1_{\alpha}$, and there exist $u_\pm \in L^2_x \mathcal{H}^1_{\alpha}$ such that
\begin{align}
\|u(t,x,\alpha) - e^{it\Delta_{x,\alpha}} u_\pm(x,\alpha)\|_{L^2_x \mathcal{H}^1_{\alpha}} \to 0, \quad \text{as } t \to \pm \infty.
\end{align}
\end{proposition}
The proofs for the above two propositions follow similarly to the cubic case, so we omit them.

Now we turn to \eqref{maineq2}. We present a local well-posedness result as follows. Consider the cubic model as an example.
\begin{equation}\label{maineq3dcubicb}
    (i\partial_t+\Delta_x)u+\partial_{\alpha}\cdot(\partial_{\alpha}u e^{-\frac{\alpha^2}{2}})u=|u|^2 u,\quad u(0,x,\alpha)=u_0.
\end{equation}   
Straightforward calculations are as follows. Duhamel Formula, the Sobolev, conservation law, and nonlinear estimate (Lemma \ref{delta}) are used. We use $e^{itH}$ for the linear propagator in \eqref{maineq3dcubicb}. For $s>1$ ($s$ can be arbitrarily close to $1$), we have
\begin{align*}
\norm{u}_{H_x^sH^1_{\alpha}} & \leq \norm{e^{it H} u_0}_{H_x^sH^1_{\alpha}}  + \norm{\int_0^t e^{i(t-s)H} \abs{u}^2 u (s) \, ds}_{H_x^sH^1_{\alpha}} \\
& \leq \norm{u_0}_{H_x^sH^1_{\alpha}} + \int_0^t \norm{ e^{i(t-s)H} \abs{u}^2 u (s) }_{H_x^sH^1_{\alpha}} \, ds\\
& \leq \norm{u_0}_{H_x^s H^1_{\alpha}} + \int_0^t \norm{ \abs{u}^2 u (s) }_{H_x^s H^1_{\alpha}} \, ds\\
& \leq \norm{u_0}_{H_x^s H^1_{\alpha}} + \int_0^t \norm{u}_{H_x^sH^1_{\alpha}}^3 \, ds \\
& \leq \norm{u_0}_{H_x^s H^1_{\alpha}} + T \norm{u}_{H_x^s H^1_{\alpha}}^3. 
\end{align*} 
By taking $T>0$ small enough, it yields
\begin{align*}
\norm{u}_{H_x^s  H^1_{\alpha}} \leq 2 \norm{u_0}_{H_x^s H^1_{\alpha}},
\end{align*}    
which implies local well-posedness.

\begin{remark}
For \eqref{maineq2}, analogous results for the 2D quintic model and other general models can be obtained in a similar way. Since the initial space is not compatible with the conservation law (Section \ref{Cons}). The global dynamics are not very clear.  See Section \ref{foc} for finite-time blow-up results in the focusing setting.     
\end{remark}

\begin{remark}[\eqref{maineq2} vs. \eqref{maineq1} in LWP theory]
Compared to \eqref{maineq1}, establishing local well-posedness for \eqref{maineq2} requires greater care. In \eqref{maineq1}, the $\mathcal{OU}$ operator admits a clean spectral resolution and ``almost commutes" well with one spatial derivative in $\alpha$ direction. In contrast, the divergence-form $\mathcal{OU}$ operator in \eqref{maineq2} breaks the derivative commutation, and the propagator does not preserve classical Sobolev spaces. To circumvent this, we rely on high regularity estimates.
\end{remark}

\section{Gaussian-weighted Morawetz Estimates}\label{Morawetz}
In this section, we establish the Gaussian-weighted interaction Morawetz estimates for both \eqref{maineq2} and \eqref{maineq1}, which are expected to be crucial ingredients for studying the long time dynamics.

These Gaussian-weighted interaction Morawetz estimates are new and of their own interests. More applications (such as scattering dynamics) are expected to be further investigated in the future. In addition, the Gaussian-weighted interaction Morawetz estimates developed here may also provide useful insights for establishing weighted interaction Morawetz inequalities in other non-Euclidean or non-compact settings.
We refer to \cite{Taobook,Iteam1,colliander2009tensor,planchon2009bilinear} and the references therein for classical (interaction) Morawetz estimates.

\subsection{Gaussian-weighted Morawetz Estimates for \eqref{maineq2}}\label{Morawetz 2}

First, we have
\begin{proposition}[Interaction Morawetz Identity]\label{prop Morawetz}

Let $u(t,x,\alpha)$ satisfies
\begin{align}
i\partial_t u+\Delta_x u + \nabla_{\alpha} \cdot (\nabla_{\alpha} u \,e^{-\frac{\alpha^2}{2}} ) 
 = |u|^p u,
\end{align}
and let $v(t,y,\beta)$ verify
\begin{align}
i\partial_t v +\Delta_y v + \nabla_{\beta} \cdot (\nabla_{\beta} v\, e^{-\frac{\beta^2}{2}} )  & = \abs{v}^p v .
\end{align}
Set a bilinear Morawetz interaction functional of the following form
\begin{align}
I_{\rho} :=  \iint_{(\R^d \times \R) \times (\R^d \times \R)} \rho (x-y) \abs{u}^2 (x , \alpha) \abs{v}^2 (y, \beta)  \, dx d\alpha dy d\beta .
\end{align}
Note here $\rho (x-y)$ depends only on $x$ and $y$, not on $\alpha$ or $\beta$, and we also assume the symmetry on $\rho$, that is, $\nabla_x \rho = - \nabla_y \rho$.

Then we have 
\begin{align}\label{eq Morawetz}
\begin{aligned}
\partial_t^2 I_{\rho} & = 4 \iint H_{\rho} (x-y) (F(u,v) (x, \alpha , y , \beta) , \overline{F} (u,v) (x, \alpha , y , \beta)) \, dxdy  \, d\alpha d \beta \\
& \quad + \int \frac{p}{p+2}  \int  \abs{v}^2 (y, \beta) (\Delta_x \rho) (x-y) \abs{u}^{p+2} (x, \alpha) \, dxdy  \omega(\alpha)  \, d\alpha d \beta \\
& \quad + \int \frac{p}{p+2} \int  \abs{u}^2 (x, \alpha) (\Delta_x \rho) (x-y) \abs{v}^{p+2} (y,\beta) \, dxdy \omega(\beta)  \, d\alpha d \beta, 
\end{aligned}
\end{align}
where
\begin{align}
F(u,v) : = \overline{v} (y,\beta) \nabla_x u(x, \alpha) + u(x, \alpha) \nabla_y \overline{v}(y,\beta) 
\end{align}
and $H_{\rho}$ is the Hessian matrix of $\rho$.

\end{proposition}

\begin{remark}
The functional $I_{\rho}$ measures the interaction between two solutions $u$ and $v$ weighted by $\rho (x-y)$, generalizing the classical momentum-based Morawetz identity to the $\mathcal{OU}$ setting. The weight \(e^{-\frac{\alpha^2}{2}}\) in \(I_\rho\) localizes interactions in the \(\alpha\)-direction, mimicking compactness. The weighted Morawetz estimate quantifies how the Gaussian drift confines nonlinear energy transfer, preventing soliton formation.   
\end{remark}

\begin{proof}[Proof of Proposition \ref{prop Morawetz}]

Recall that
\begin{align}\label{eq 2}
\begin{aligned}
 \partial_t u & = -i [\abs{u}^p u  - \Delta_x u - \nabla_{\alpha} \cdot (\nabla_{\alpha} u e^{-\frac{\alpha^2}{2}} ) ] \\
\partial_t \overline{u} & = i [\abs{u}^p \overline{u}  - \Delta_x \overline{u}  - \nabla_{\alpha} \cdot (\nabla_{\alpha} \overline{u} e^{-\frac{\alpha^2}{2}} )  ]  
\end{aligned}
\end{align}
then using the equation, we write
\begin{align}
\partial_t \abs{u}^2 (t,x, \alpha)  & = 2 \re u_t \overline{u} (t,x,\alpha) \\
& = 2 \re [-i (\abs{u}^p u  - \Delta_x u - \nabla_{\alpha} \cdot (\nabla_{\alpha} u e^{-\frac{\alpha^2}{2}} ) ) \overline{u}] \\
& = 2 \im  [(\abs{u}^p u  - \Delta_x u - \nabla_{\alpha} \cdot (\nabla_{\alpha} u e^{-\frac{\alpha^2}{2}} )) \overline{u}] \\
& = 2 \im  [(- \Delta_x u - \nabla_{\alpha} \cdot (\nabla_{\alpha} u e^{-\frac{\alpha^2}{2}} )) \overline{u}]  \\
& = 2 \im [-\nabla_x \cdot (\overline{u} \nabla_x u) ]    + 2 \im [-\nabla_{\alpha} \cdot ( \nabla_{\alpha} u e^{-\frac{\alpha^2}{2}}) \overline{u}] . \label{eq 4}
\end{align}

We now take the first derivative of $I_{\rho}$ with respect to time $t$. Using the equation and integration by parts, we have
\begin{align}
\partial_t I_{\rho} & = \iint \rho (x-y) \sq{\partial_t \abs{u}^2 (x,\alpha) } \abs{v}^2 (y, \beta)  \, dx d\alpha dy d\beta  + \iint \rho (x-y)  \abs{u}^2  (x,\alpha)  \sq{\partial_t \abs{v}^2 (y, \beta)  }   \, dx d\alpha dy d\beta \\
& = 2 \im \iint \rho (x-y) \sq{- \nabla_x (\overline{u} \nabla_x u)   - \nabla_{\alpha} \cdot ( \nabla_{\alpha} u e^{-\frac{\alpha^2}{2}}) \overline{u}} (x,\alpha) \abs{v}^2 (y, \beta)  \, dx d\alpha dy d\beta  \\
& \quad + 2 \im \iint \rho (x-y) \abs{u}^2  (x,\alpha)  \sq{- \nabla_y (\overline{v} \nabla_y v) - \nabla_{\beta} \cdot(\nabla_{\beta} v e^{-\frac{\beta^2}{2}}) \overline{v} } (y,\beta)  \, dx d\alpha dy d\beta \\
& = 2 \im \iint \nabla_x \rho \sq{ (\overline{u} \nabla_x u) (x,\alpha) }  \abs{v}^2 (y,\beta)  \, dx d\alpha dy d\beta  \\
& \quad + 2 \im \iint \rho \sq{ \abs{ \nabla_{\alpha} u}^2  (x,\alpha) e^{-\frac{\alpha^2}{2}} } \abs{v}^2 (y,\beta)  \, dx d\alpha dy d\beta \\
& \quad + 2 \im \iint \nabla_y \rho \abs{u}^2  (x,\alpha)   \sq{   (\overline{v} \nabla_y v) (y,\beta) }  \, dx d\alpha dy d\beta   \\
& \quad + 2 \im \iint \rho  \abs{u}^2  (x,\alpha)  \sq{ \abs{\nabla_{\beta} v }^2 (y,\beta) e^{-\frac{\beta^2}{2}}} \, dx d\alpha dy d\beta .
\end{align}

Notice that $\nabla_x \rho = - \nabla_y \rho$, we have
\begin{align}
\partial_t I_{\rho} & = 2 \im \iint \nabla_x \rho \sq{ (\overline{u} \nabla_x u) (x,\alpha)  \abs{v}^2 (y,\beta)  -\abs{u}^2  (x,\alpha)  (\overline{v} \nabla_y v) (y,\beta) } \, dx d\alpha dy d\beta \\
& \quad  + 2 \im \iint \rho  \sq{ \abs{\nabla_{\alpha} u }^2  (x,\alpha) e^{-\frac{\alpha^2}{2}}  \abs{v}^2 (y,\beta)  + \abs{u}^2  (x,\alpha)  \abs{ \nabla_{\beta} v}^2 (y,\beta) e^{-\frac{\beta^2}{2}}}  \, dx d\alpha dy d\beta  \\
& = 2 \im \iint \nabla_x \rho \sq{ (\overline{u} \nabla_x u) (x,\alpha)  \abs{v}^2 (y,\beta)  -\abs{u}^2  (x,\alpha)  (\overline{v} \nabla_y v) (y,\beta) } \, dx d\alpha dy d\beta \label{eq 1}.
\end{align}
Note that the second term in the integral above is real, hence the integral is equal to zero.

Next, we will compute the second derivative of $I_{\rho}$ in $t$, which is the derivative of \eqref{eq 1}. 
\begin{align}
\partial_t^2 I_{\rho} & = 2 \im \partial_t \iint \nabla_x \rho \sq{ (\overline{u} \nabla_x u) (x,\alpha)  \abs{v}^2 (y,\beta)  -\abs{u}^2  (x,\alpha)  (\overline{v} \nabla_y v) (y,\beta) } \, dx d\alpha dy d\beta \\ 
& = : \int J  \, d\alpha d \beta. \label{eq 3}
\end{align}

Recall that in Section 4.3 in \cite{PV}, the following Morawetz computation for NLS  on $J$ term is given by
\begin{align}
J & = 4 \int H_{\rho} (x-y) (F(u,v) (x,y) , \overline{F} (u,v) (x,y)) \, dxdy \\
& \quad + \frac{p}{p+2}  \int  \abs{v}^2 (y) (\Delta_x \rho) (x-y) \abs{u}^{p+2} (x) \, dxdy \\
& \quad +  \frac{p}{p+2} \int  \abs{u}^2 (x) (\Delta_x \rho) (x-y) \abs{v}^{p+2} (y) \, dxdy 
\end{align}
where
\begin{align}
F(u,v) : = \overline{v} (y) \nabla_x u(x) + u(x) \nabla_y \overline{v}(y) 
\end{align}
and $H_{\rho}$ is the Hessian matrix of $\rho$, hence  symmetric.

Then back to \eqref{eq 3}
\begin{align}
\eqref{eq 3} = \partial_t^2 I_{\rho} & =  \int J  \, d\alpha d \beta  \\
& = 4 \iint H_{\rho} (x-y) (F(u,v) (x, \alpha , y , \beta) , \overline{F} (u,v) (x, \alpha , y , \beta)) \, dxdy d\alpha d \beta \\
& \quad + \int  \frac{p}{p+2}  \int  \abs{v}^2 (y, \beta) (\Delta_x \rho) (x-y) \abs{u}^{p+2} (x, \alpha) \, dxdy  d\alpha d \beta \\
& \quad + \int \frac{p}{p+2} \int  \abs{u}^2 (x, \alpha) (\Delta_x \rho) (x-y) \abs{v}^{p+2} (y,\beta) \, dxdy d\alpha d \beta, 
\end{align}
which completes the proof of Proposition \ref{prop Morawetz}.
\end{proof}

\begin{corollary}[Interaction Morawetz Estimate]\label{cor Morawetz}
For $\rho (x-y) = \abs{x-y}$ or $\inner{x-y}$, 
\begin{align}
&  \int \iint  \frac{p}{p+2}  \abs{u}^2 (y) (\Delta_x \rho) (x-y) \abs{u}^{p+2} (x) \, dxdy d\alpha d \beta dt\\
& \quad + \int \iint \frac{p}{p+2}   \abs{u}^2 (x) (\Delta_x \rho) (x-y) \abs{u}^{p+2} (y) \, dxdy d\alpha d \beta dt  \\
& \lesssim \norm{u}_{L_x^2 L_\alpha^2}^3  \norm{u}_{\dot{H}_x^1 L_\alpha^2} .
\end{align}
\end{corollary}
\begin{remark}
For \(\rho = |x-y|\), our estimate reduces to the waveguide case (\cite{TV2}, Cor. 5.2) when \(e^{-\frac{\alpha^2}{2}} \equiv 1\).    
\end{remark}
\begin{proof}[Proof of Corollary \ref{cor Morawetz}]
Notice that in Proposition \ref{prop Morawetz}, when taking $u=v$ and choosing $\rho (x-y) = \abs{x-y}$ or $\inner{x-y}$, we see that 
\begin{align}
\iint H_{\rho} (x-y) (F(u,v) (x,y) , \overline{F} (u,v) (x,y)) \, dxdy d\alpha d \beta \geq 0,
\end{align}
hence we ignore this term from \eqref{eq Morawetz}.

Under the choice of $\rho (x-y) = \abs{x-y}$ or $\inner{x-y}$, we compute
\begin{align}
\Delta_x \abs{x-y} &  = \Delta_y \abs{x-y}= \frac{d-1} { \abs{x-y}}  \\
\Delta_x \inner{x-y} & = \Delta_y \inner{x-y}= \frac{d-1} { \inner{x-y}} + \frac{1}{\inner{x-y}^3}.
\end{align}

Recall \eqref{eq 1}, a straightforward computation yields 
\begin{align}
\abs{\partial_t I_{\rho}} \lesssim \norm{u}_{L_x^2 L_\alpha^2}^3  \norm{u}_{\dot{H}_x^1 L_\alpha^2}. 
\end{align}
Combining the above information, we have
\begin{align}
&  \int \iint \frac{p}{p+2}  \abs{u}^2 (y) (\Delta_x \rho) (x-y) \abs{u}^{p+2} (x) \, dxdy d\alpha d \beta dt\\
& \quad + \int \iint \frac{p}{p+2}   \abs{u}^2 (x) (\Delta_x \rho) (x-y) \abs{u}^{p+2} (y) \, dxdy d\alpha d \beta dt  \\
& \lesssim \norm{u}_{L_x^2 L_\alpha^2}^3  \norm{u}_{\dot{H}_x^1 L_\alpha^2} .
\end{align}

This completes the proof of Corollary \ref{cor Morawetz}.
\end{proof}

This completes the interaction Morawetz estimates for \eqref{maineq2}. 

\subsection{Gaussian-weighted Morawetz Estimates for \eqref{maineq1}}\label{Morawetz 2}
We now turn to \eqref{maineq1}. We have,
\begin{proposition}[Interaction Morawetz Identity]\label{prop Morawetz2}

Let $u(t,x,\alpha)$ that satisfies 
\begin{align}
(i\partial_t  + \Delta_x + \Delta_{\alpha} - \alpha \cdot \nabla_{\alpha} ) u & = \omega(\alpha) \abs{u}^p u 
\end{align}
and $v(t,y,\beta)$ that verifies
\begin{align}
(i\partial_t  + \Delta_y + \Delta_{\beta} - \beta \cdot \nabla_{\beta} ) v & =  \omega(\beta) \abs{v}^p v.
\end{align}
Set a bilinear Morawetz interaction functional of the following form
\begin{align}
I_{\rho} :=  \iint_{(\R^d \times \R) \times (\R^d \times \R)} \rho (x-y) \abs{u}^2 (x , \alpha) \abs{v}^2 (y, \beta) e^{-\frac{\alpha^2}{2}}  e^{-\frac{\beta^2}{2}} \, dx d\alpha dy d\beta .
\end{align}
Note here $\rho (x-y)$ depends only on $x$ and $y$, not on $\alpha$ or $\beta$, and we also assume the symmetry on $\rho$, that is, $\nabla_x \rho = - \nabla_y \rho$.

Then we have 
\begin{align}\label{eq Morawetz2}
\begin{aligned}
\partial_t^2 I_{\rho} & = 4 \iint H_{\rho} (x-y) (F(u,v) (x, \alpha , y , \beta) , \overline{F} (u,v) (x, \alpha , y , \beta)) \, dxdy e^{- \frac{\alpha^2}{2} -\frac{\beta^2}{2}}  \, d\alpha d \beta \\
& \quad + \int \frac{p}{p+2}  \int  \abs{v}^2 (y, \beta) (\Delta_x \rho) (x-y) \abs{u}^{p+2} (x, \alpha) \, dxdy  \omega(\alpha) e^{- \frac{\alpha^2}{2} -\frac{\beta^2}{2}} \, d\alpha d \beta \\
& \quad + \int \frac{p}{p+2} \int  \abs{u}^2 (x, \alpha) (\Delta_x \rho) (x-y) \abs{v}^{p+2} (y,\beta) \, dxdy \omega(\beta) e^{- \frac{\alpha^2}{2} -\frac{\beta^2}{2}} \, d\alpha d \beta 
\end{aligned}
\end{align}
where
\begin{align}
F(u,v) : = \overline{v} (y,\beta) \nabla_x u(x, \alpha) + u(x, \alpha) \nabla_y \overline{v}(y,\beta) 
\end{align}
and $H_{\rho}$ is the Hessian matrix of $\rho$.

\end{proposition}

\begin{proof}[Proof of Proposition \ref{prop Morawetz2}]

Recall that
\begin{align}\label{eq 22}
\begin{aligned}
 \partial_t u & = -i (\abs{u}^p u \omega(\alpha)  - \Delta_x u - \Delta_{\alpha} u + \alpha \cdot \nabla_{\alpha} u) \\
\partial_t \overline{u} & = i (\abs{u}^p \overline{u} \omega(\alpha) - \Delta_x \overline{u}  - \Delta_{\alpha} \overline{u} + \alpha \cdot \nabla_{\alpha} \overline{u})   
\end{aligned}
\end{align}
then using the equation we write
\begin{align}
\partial_t \abs{u}^2 (t,x, \alpha) e^{-\frac{\alpha^2}{2}}  & = 2 \re u_t \overline{u} (t,x,\alpha) e^{-\frac{\alpha^2}{2}}   \\
& = 2 \re [-i (\abs{u}^p u \omega(\alpha) - \Delta_x u - \Delta_{\alpha} u + \alpha \cdot \nabla_{\alpha} u) \overline{u}] e^{-\frac{\alpha^2}{2}}  \\
& = 2 \im  [(\abs{u}^p u \omega(\alpha) - \Delta_x u - \Delta_{\alpha} u + \alpha \cdot \nabla_{\alpha} u) \overline{u}] e^{-\frac{\alpha^2}{2}}   \\
& = 2 \im  [(- \Delta_x u - \Delta_{\alpha} u + \alpha \cdot \nabla_{\alpha} u) \overline{u}] e^{-\frac{\alpha^2}{2}}  \\
& = 2 \im [-\nabla_x \cdot ( \overline{u} \nabla_x u) e^{-\frac{\alpha^2}{2}} ]  + 2 \im  [ - \Delta_{\alpha} u + \alpha \cdot \nabla_{\alpha} u) \overline{u} ] e^{-\frac{\alpha^2}{2}}  \\ 
& = 2 \im [-\nabla_x \cdot (\overline{u} \nabla_x u) e^{-\frac{\alpha^2}{2}} ]    + 2 \im [-\nabla_{\alpha} \cdot ( \nabla_{\alpha} u e^{-\frac{\alpha^2}{2}}) \overline{u}] .
\end{align}

We now take the first derivative of $I_{\rho}$ with respect to time $t$. Using the equation and integration by parts, we have
\begin{align}
\partial_t I_{\rho} & = \iint \rho (x-y) \sq{\partial_t \abs{u}^2 (x,\alpha)e^{-\frac{\alpha^2}{2}} } \abs{v}^2 (y, \beta) e^{-\frac{\beta^2}{2}} \, dx d\alpha dy d\beta \\
& \quad + \iint \rho (x-y)  \abs{u}^2  (x,\alpha) e^{-\frac{\alpha^2}{2}}  \sq{\partial_t \abs{v}^2 (y, \beta) e^{-\frac{\beta^2}{2}} }   \, dx d\alpha dy d\beta \\
& = 2 \im \iint \rho (x-y) \sq{- \nabla_x (\overline{u} \nabla_x u)  e^{-\frac{\alpha^2}{2}}  - \nabla_{\alpha} \cdot ( \nabla_{\alpha} u e^{-\frac{\alpha^2}{2}}) \overline{u}} (x,\alpha) \abs{v}^2 (y, \beta) e^{-\frac{\beta^2}{2}} \, dx d\alpha dy d\beta  \\
& \quad + 2 \im \iint \rho (x-y) \abs{u}^2  (x,\alpha) e^{-\frac{\alpha^2}{2}} \sq{- \nabla_y (\overline{v} \nabla_y v) e^{-\frac{\beta^2}{2}} - \nabla_{\beta} \cdot(\nabla_{\beta} v e^{-\frac{\beta^2}{2}}) \overline{v} } (y,\beta)  \, dx d\alpha dy d\beta \\
& = 2 \im \iint \nabla_x \rho \sq{ (\overline{u} \nabla_x u) (x,\alpha) e^{-\frac{\alpha^2}{2}}}  \abs{v}^2 (y,\beta) e^{-\frac{\beta^2}{2}} \, dx d\alpha dy d\beta  \\
& \quad + 2 \im \iint \rho \sq{ \abs{ \nabla_{\alpha} u}^2  (x,\alpha) e^{-\frac{\alpha^2}{2}} } \abs{v}^2 (y,\beta) e^{-\frac{\beta^2}{2}} \, dx d\alpha dy d\beta \\
& \quad + 2 \im \iint \nabla_y \rho \abs{u}^2  (x,\alpha) e^{-\frac{\alpha^2}{2}}  \sq{   (\overline{v} \nabla_y v) (y,\beta) e^{-\frac{\beta^2}{2}}}  \, dx d\alpha dy d\beta   \\
& \quad + 2 \im \iint \rho  \abs{u}^2  (x,\alpha) e^{-\frac{\alpha^2}{2}} \sq{ \abs{\nabla_{\beta} v }^2 (y,\beta) e^{-\frac{\beta^2}{2}}} \, dx d\alpha dy d\beta .
\end{align}

Notice that $\nabla_x \rho = - \nabla_y \rho$, we have
\begin{align}
\partial_t I_{\rho} & = 2 \im \iint \nabla_x \rho \sq{ (\overline{u} \nabla_x u) (x,\alpha) e^{-\frac{\alpha^2}{2}} \abs{v}^2 (y,\beta)e^{-\frac{\beta^2}{2}}  -\abs{u}^2  (x,\alpha) e^{-\frac{\alpha^2}{2}} (\overline{v} \nabla_y v) (y,\beta) e^{-\frac{\beta^2}{2}}} \, dx d\alpha dy d\beta \\
& \quad  + 2 \im \iint \rho  \sq{ \abs{\nabla_{\alpha} u }^2  (x,\alpha) e^{-\frac{\alpha^2}{2}}  \abs{v}^2 (y,\beta) e^{-\frac{\beta^2}{2}} + \abs{u}^2  (x,\alpha) e^{-\frac{\alpha^2}{2}} \abs{ \nabla_{\beta} v}^2 (y,\beta) e^{-\frac{\beta^2}{2}}}  \, dx d\alpha dy d\beta  \\
 & = 2 \im \iint \nabla_x \rho \sq{ (\overline{u} \nabla_x u) (x,\alpha) e^{-\frac{\alpha^2}{2}} \abs{v}^2 (y,\beta)e^{-\frac{\beta^2}{2}}  -\abs{u}^2  (x,\alpha) e^{-\frac{\alpha^2}{2}} (\overline{v} \nabla_y v) (y,\beta) e^{-\frac{\beta^2}{2}}} \, dx d\alpha dy d\beta. \label{eq 12}
\end{align}
Note that the second term in the integral above is real, hence the integral is equal to zero.

Next, we will compute the second derivative of $I_{\rho}$ in $t$, which is the derivative of \eqref{eq 12}. 
\begin{align}
\partial_t^2 I_{\rho} & = 2 \im \partial_t \iint \nabla_x \rho \sq{(\overline{u} \nabla_x u) (x,\alpha)  \abs{v}^2 (y,\beta) -\abs{u}^2  (x,\alpha)  (\overline{v} \nabla_y v) (y,\beta)} e^{-\frac{\alpha^2}{2}-\frac{\beta^2}{2}} \, dx dy d\alpha  d\beta \\ 
& = : \int J e^{- \frac{\alpha^2}{2} -\frac{\beta^2}{2}}  \, d\alpha d \beta.\label{eq 32}
\end{align}

Recall that in Section 4.3 in \cite{PV}, the following Morawetz computation for NLS  on $J$ term is given by
\begin{align}
J & = 4 \int H_{\rho} (x-y) (F(u,v) (x,y) , \overline{F} (u,v) (x,y)) \, dxdy \\
& \quad + \omega(\alpha) \frac{p}{p+2}  \int  \abs{v}^2 (y) (\Delta_x \rho) (x-y) \abs{u}^{p+2} (x) \, dxdy \\
& \quad + \omega(\beta) \frac{p}{p+2} \int  \abs{u}^2 (x) (\Delta_x \rho) (x-y) \abs{v}^{p+2} (y) \, dxdy, 
\end{align}
where
\begin{align}
F(u,v) : = \overline{v} (y) \nabla_x u(x) + u(x) \nabla_y \overline{v}(y) 
\end{align}
and $H_{\rho}$ is the Hessian matrix of $\rho$, hence  symmetric.

Then back to \eqref{eq 32}
\begin{align}
\eqref{eq 32} = \partial_t^2 I_{\rho} & =  \int J e^{- \frac{\alpha^2}{2} -\frac{\beta^2}{2}}  \, d\alpha d \beta  \\
& = 4 \iint H_{\rho} (x-y) (F(u,v) (x, \alpha , y , \beta) , \overline{F} (u,v) (x, \alpha , y , \beta)) \, dxdy e^{- \frac{\alpha^2}{2} -\frac{\beta^2}{2}}  \, d\alpha d \beta \\
& \quad + \int \frac{p}{p+2}  \int  \abs{v}^2 (y, \beta) (\Delta_x \rho) (x-y) \abs{u}^{p+2} (x, \alpha) \, dxdy  \omega(\alpha) e^{- \frac{\alpha^2}{2} -\frac{\beta^2}{2}} \, d\alpha d \beta \\
& \quad + \int \frac{p}{p+2} \int  \abs{u}^2 (x, \alpha) (\Delta_x \rho) (x-y) \abs{v}^{p+2} (y,\beta) \, dxdy \omega(\beta) e^{- \frac{\alpha^2}{2} -\frac{\beta^2}{2}} \, d\alpha d \beta, 
\end{align}
which completes the proof of Proposition \ref{prop Morawetz2}.
\end{proof}

\begin{corollary}[Interaction Morawetz estimate]\label{cor Morawetz2}
For $\rho (x-y) = \abs{x-y}$ or $\inner{x-y}$, 
\begin{align}
&  \int \iint \frac{p}{p+2}  \abs{u}^2 (y) (\Delta_x \rho) (x-y) \abs{u}^{p+2} (x) \, dxdy \omega(\alpha) e^{- \frac{\alpha^2}{2} -\frac{\beta^2}{2}}  \, d\alpha d \beta dt\\
& \quad + \int \iint \frac{p}{p+2}   \abs{u}^2 (x) (\Delta_x \rho) (x-y) \abs{u}^{p+2} (y) \, dxdy \omega(\beta) e^{- \frac{\alpha^2}{2} -\frac{\beta^2}{2}} \, d\alpha d \beta dt  \\
& \lesssim \norm{u}_{L_x^2 \mathcal{L}_\alpha^2}^3  \norm{u}_{\dot{H}_x^1 \mathcal{L}_\alpha^2} .
\end{align}
    
\end{corollary}

\begin{proof}[Proof of Corollary \ref{cor Morawetz2}]
Notice that in Proposition \ref{prop Morawetz2}, when taking $u=v$ and choosing $\rho (x-y) = \abs{x-y}$ or $\inner{x-y}$, we see that 
\begin{align}
\iint H_{\rho} (x-y) (F(u,v) (x,y) , \overline{F} (u,v) (x,y)) \, dxdy e^{- \frac{\alpha^2}{2} -\frac{\beta^2}{2}}  \, d\alpha d \beta \geq 0
\end{align}
hence we ignore this term from \eqref{eq Morawetz2}.

Under the choice of $\rho (x-y) = \abs{x-y}$ or $\inner{x-y}$, we compute
\begin{align}
\Delta_x \abs{x-y} &  = \Delta_y \abs{x-y}= \frac{d-1} { \abs{x-y}}  \\
\Delta_x \inner{x-y} & = \Delta_y \inner{x-y}= \frac{d-1} { \inner{x-y}} + \frac{1}{\inner{x-y}^3}.
\end{align}

Recall \eqref{eq 12}, a straightforward computation yields 
\begin{align}
\abs{\partial_t I_{\rho}} \lesssim \norm{u}_{L_x^2 \mathcal{L}_\alpha^2}^3  \norm{u}_{\dot{H}_x^1 \mathcal{L}_\alpha^2},
\end{align}
where the notation $\mathcal{L}$ is defined as in \eqref{eq L^2}.

Combining the above information, we have
\begin{align}
&  \int \iint \frac{p}{p+2}  \abs{u}^2 (y) (\Delta_x \rho) (x-y) \abs{u}^{p+2} (x) \, dxdy \omega(\alpha) e^{- \frac{\alpha^2}{2} -\frac{\beta^2}{2}}  \, d\alpha d \beta dt\\
& \quad + \int \iint \frac{p}{p+2}   \abs{u}^2 (x) (\Delta_x \rho) (x-y) \abs{u}^{p+2} (y) \, dxdy \omega(\beta) e^{- \frac{\alpha^2}{2} -\frac{\beta^2}{2}} \, d\alpha d \beta dt  \\
& \lesssim \norm{u}_{L_x^2 \mathcal{L}_\alpha^2}^3  \norm{u}_{\dot{H}_x^1 \mathcal{L}_\alpha^2} .
\end{align}

We finish the proof of Corollary \ref{cor Morawetz2}.
\end{proof}

This completes the interaction Morawetz estimates for \eqref{maineq1}. The Gaussian-weighted interaction Morawetz estimates for both \eqref{maineq2} and \eqref{maineq1} are now complete.\footnote{As we can see directly, the differences for the two cases are the weighted norms for the $\mathcal{OU}$-direction. This is also consistent to the Strichartz estimates in Section \ref{Stri}.} Here are a few remarks as below. 

\begin{remark}
The Gaussian weights in the $\mathcal{OU}$ direction provide localization mechanisms analogous to compactness in waveguides. Our adaptation of Morawetz estimates to this setting represents a novel approach to controlling long-range interactions in non-compact geometries.
\end{remark}

\begin{remark}
We note that, the interaction Morawetz estimates established in this section for both \eqref{maineq2} and \eqref{maineq1} are compatible with the conservation laws\footnote{In short, the bounds on the right hand side are energy bounds.}. They are expected to be useful for further investigating the long time dynamics.  
\end{remark}

\begin{remark}[A quick remark on the scattering]
We underline that arguing as in \cite{antonelli2015scattering}, where NLS with a partially confining potential is studied, the following version of interaction Morawetz estimate would be helpful:
\begin{equation}\label{mora}
    \int_{\mathbb{R}} \int_{\mathbb{R}^d}  \left| |D_x|^{\frac{(3-d)}{2}} ( \int_{\mathbb{R}} |u(t,x,y)|^2 d\alpha) \right|^2 dx dt<\infty. 
\end{equation}
Via the Sobolev embedding, one can deduce some {\it a priori} bounds in the case $d = 1, 2, 3$, which are sufficient to deduce scattering for the subcritical case (see the computations in \cite{antonelli2015scattering} in the case of a partially confining potential and see Remark 1.9 in \cite{TV2}).

However, for our case, a variant of \eqref{mora} is not obtainable. That's a reason that the large data scattering in energy space remains open for our models. We leave it as a future goal.
\end{remark}

\section{Finite-time Blow-up Solutions in the Focusing Setting}\label{foc}
In this section, we study the focusing analogue of \eqref{maineq2}. We will show finite-time blow-up solutions in the focusing setting. 

First, we recall the focusing NLS with $\mathcal{OU}$ operator as follows,
\begin{align}\label{eq focusing}
i\partial_t u+\Delta_x u + \nabla_{\alpha} \cdot (\nabla_{\alpha} u\, e^{-\frac{\alpha^2}{2}} ) = - \abs{u}^p u,,  \qquad (t,x, \alpha) \in I \times \R^d \times \R,
\end{align}
where $I=[0,T)$ indicates the maximal lifespan of the solution. 

Then we present the proof of Theorem \ref{mainthm3} as follows.
\begin{proof}[Proof of Theorem \ref{mainthm3}]
The strategy in this proof follows from the so-called virial identity. In fact, we wish to derive a virial identity by taking the second time derivative of the following associated virial potential\footnote{As we can see, we put the weight $|x|^2$ only on the Euclidean direction.}
\begin{align}\label{eq Virial}
V(t) : = \int \abs{x}^2 \abs{u}^2 (t,x,\alpha) \, dx d\alpha .
\end{align}
Then we prove that
\begin{align}
V''(t) \leq C E(u_0)
\end{align}
with $C >0$. Then when assuming $E(u_0) <0$, the positivity of $V(t)$ yields the conclusion via a standard convexity argument. 

Let us remark that in the definition of \eqref{eq Virial}, we used the weight $\abs{x}^2$ (depending only on $x \in \R^d$, not $\alpha$) instead of $\abs{x}^2 + \abs{\alpha}^2$. This mimics what we did in the Morawetz estimate potential (see Proposition \ref{prop Morawetz}).

First, recall from \eqref{eq 4}
\begin{align}
\partial_t \abs{u}^2 (t,x, \alpha) & = 2 \im [-\nabla_x \cdot (\overline{u} \nabla_x u) ]  + 2 \im [-\nabla_{\alpha} \cdot ( \nabla_{\alpha} u e^{-\frac{\alpha^2}{2}}) \overline{u}] .
\end{align}
Then we compute the first time derivative of $V$, and obtain
\begin{align}
\partial_t V(t) & = \partial_t \int \abs{x}^2 \abs{u}^2 (t,x,\alpha) \, dx d\alpha \\
& = \int \abs{x}^2 \partial_t \abs{u}^2 \, dxd\alpha \\
& = 2 \im \int \abs{x}^2 [-\nabla_x \cdot (\overline{u} \nabla_x u) -\nabla_{\alpha} \cdot ( \nabla_{\alpha} u e^{-\frac{\alpha^2}{2}}) \overline{u}] \, dx d\alpha \\
& = 2 \im \int 2 x \cdot \nabla_x u \overline{u}  + \abs{x}^2 \abs{\nabla_{\alpha} u}^2 e^{-\frac{\alpha^2}{2}} \, dx d\alpha \\
& = 4 \im \int x \cdot \nabla_x u \overline{u} \, dx d\alpha .
\end{align}

Then the second time derivative of $V(t)$ is given by
\begin{align}
\frac{1}{4} \partial_t^2 V(t) & = \partial_t \im \int x \cdot \nabla_x u \bar{u} \, dx d\alpha \\
& = \im \int x \cdot \nabla_x u \partial_t \bar{u} \, dx d\alpha + \im \int x \cdot \nabla_x \partial_t u  \bar{u} \, dx d\alpha \\
& = -\im \int \partial_t u x \cdot \nabla_x \bar{u} \, dx d\alpha - \im \int \partial_t u \nabla_x \cdot (x \bar{u}) \, dx d\alpha\\
& = -\im \int \partial_t u \overline{(2x \cdot \nabla_x u + du)} \, dx d\alpha\\
& = 2 \re \int i \partial_t u (\overline{\frac{d}{2} u + x \cdot \nabla_x u}) \, dx d\alpha\\
& = -2 \re \int (\Delta_x u + \nabla_{\alpha} \cdot (\nabla_{\alpha} u\, e^{-\frac{\alpha^2}{2}} ) + \abs{u}^{p}u ) (\overline{\frac{d}{2} u + x \cdot \nabla_x u}) \, dx d\alpha \\
& : = I_1 + I_2 + I_3 .
\end{align}

Next, we work on these three terms separately,
\begin{align}
I_3 &= - 2\re \int \abs{u}^{p}u (\overline{\frac{d}{2} u + x \cdot \nabla_x u}) \, dx  d\alpha \\
& = - d \int \abs{u}^{p+2} \, dx  d\alpha + \frac{2d}{p+2} \int \abs{u}^{p+2} \, dx  d\alpha\\
& = - \frac{dp}{p+2} \int \abs{u}^{p+2} \, dx  d\alpha .
\end{align}

For $I_1$, by Pohozaev identity, we have
\begin{align}
-\re \int \Delta_x u (\frac{d}{2} u + x \cdot \nabla_x u) \, dx =\int \abs{\nabla_x u}^2 \, dx ,
\end{align}
then
\begin{align}
I_1 = 2 \int \abs{\nabla_x u}^2 \, dx d \alpha .
\end{align}

Next, for $I_2$,
\begin{align}
I_2 & = -2 \re \int \nabla_{\alpha} \cdot (\nabla_{\alpha} u\, e^{-\frac{\alpha^2}{2}} ) (\overline{\frac{d}{2} u + x \cdot \nabla_x u}) \, dx d\alpha \\
& = 2 \re \int (\nabla_{\alpha} u\, e^{-\frac{\alpha^2}{2}} ) \cdot (\overline{\frac{d}{2} \nabla_{\alpha} u + x \cdot \nabla_x \nabla_{\alpha} u}) \, dx d\alpha \\
& = d \re \int \abs{\nabla_{\alpha} u}^2 e^{-\frac{\alpha^2}{2}} \, dx d\alpha + 2 \re \int (\nabla_{\alpha} u\, e^{-\frac{\alpha^2}{2}} ) \cdot (\overline{x \cdot \nabla_x \nabla_{\alpha} u}) \, dx d\alpha .
\end{align}
Notice that the second term above can be rewritten as
\begin{align}
2 \re \int (\nabla_{\alpha} u\, e^{-\frac{\alpha^2}{2}} ) \cdot (\overline{x \cdot \nabla_x \nabla_{\alpha} u}) \, dx d\alpha = - 2 d\re \int \abs{\nabla_{\alpha} u}^2 e^{-\frac{\alpha^2}{2}} \, dx d\alpha - 2 \re \int (\nabla_x \nabla_{\alpha} u ) e^{-\frac{\alpha^2}{2}} \overline{x \nabla_{\alpha} u} \, dx d\alpha \end{align}
which implies 
\begin{align}
2 \re \int (\nabla_{\alpha} u\, e^{-\frac{\alpha^2}{2}} ) \cdot (\overline{x \cdot \nabla_x \nabla_{\alpha} u}) \, dx d\alpha & = - d\re \int \abs{\nabla_{\alpha} u}^2 e^{-\frac{\alpha^2}{2}} \, dx d\alpha ,
\end{align}
hence
\begin{align}
I_2 & = d \re \int \abs{\nabla_{\alpha} u}^2 e^{-\frac{\alpha^2}{2}} \, dx d\alpha - d\re \int \abs{\nabla_{\alpha} u}^2 e^{-\frac{\alpha^2}{2}} \, dx d\alpha  = 0 .
\end{align}

Therefore,
\begin{align}
\frac{1}{16} \partial_t^2 V(t) & =\frac{1}{16} \partial_t \im \int x \cdot \nabla u \bar{u} \, dx d\alpha \\
& = \frac{1}{2} \int_{\R^d \times \R} \abs{\nabla u}^2 \, dx d\alpha - \frac{dp}{4(p+2)} \int_{\R^d \times \R} \abs{u}^{p+2} \, dx d\alpha \\
& = \frac{1}{2} \int_{\R^d \times \R} \abs{\nabla u}^2 \, dx d\alpha - \frac{1}{p+2} \int_{\R^d \times \R} \abs{u}^{p+2} \, dx d\alpha - \int_{\R^d \times \R} \abs{u}^{p+2} \sq{\frac{dp}{4(p+2)} - \frac{1}{p+2}} \, dx d\alpha \\
& = E(u)(0) - \int_{\R^d \times \R} \frac{1}{2} \abs{\nabla_{\alpha} u}^2 e^{-\frac{\alpha^2}{2}} \, dx d\alpha - \int_{\R^d \times \R} \abs{u}^{p+2} \frac{dp-4}{4(p+2)} \, dx d\alpha .
\end{align}
Thus, if the initial energy is negative, we see that the second time derivative of $V(t)$ (a non-negative quantity) is negative, hence there exist finite-time blow-up solutions via a standard virial identity argument. 

The proof of Theorem \ref{mainthm3} is now complete.
\end{proof}

\begin{remark}
A partial variant of virial identities plays a significant role in the proof. Although the $\mathcal{OU}$ operator lacks translation invariance, the virial method remains applicable due to the preservation of Gaussian-weighted mass and momentum structure.  One may also study the analogous problem for \eqref{maineq1}. We leave it for interested readers. 
\end{remark}

\section{Summary and Remarks}\label{7}

Finally, we make a summary and give final remarks. These include a comparison between \eqref{maineq2} and \eqref{maineq1}
in Section \ref{sec:8.1}, some remarks about NLS on mixed geometry in Section \ref{sec:8.2} and some future directions in Section \ref{sec:8.3}. Our results provide a first step toward understanding NLS dynamics under $\mathcal{OU}$-type confinement and open the door to further investigations.

\subsection{Comparison between \eqref{maineq2} and \eqref{maineq1}}\label{sec:8.1}

We now comment further on the analytic and geometric comparison between \eqref{maineq2} and \eqref{maineq1}, both of which model mixed-type dispersion through the $\mathcal{OU}$ operator. From a structural point of view:

\begin{itemize}
  \item \textbf{\eqref{maineq2}} employs the divergence form of the $\mathcal{OU}$ operator: 
  \begin{align}
  \nabla_\alpha \cdot (e^{-\frac{\alpha^2}{2}} \nabla_\alpha u),  
  \end{align}
  which is closely tied to Gaussian-weighted energy identities and yields a natural conservation law. However, it lacks commutativity with derivatives, complicating Strichartz analysis and function space choices. Moreover, the lack of spectral diagonalization prevents straightforward modal decomposition.

  \item \textbf{\eqref{maineq1}} adopts the non-divergence form:
  \begin{align}
  (\Delta_\alpha - \alpha \cdot \nabla_\alpha) u,
  \end{align}
  which has a discrete Hermite spectral resolution and behaves more like a compact direction in the waveguide setting. This structure allows efficient application of spectral and Strichartz techniques. On the other hand, \eqref{maineq1} requires a careful choice of nonlinearity weight (e.g., $e^{-p\frac{\alpha^2}{2}}$) to maintain energy balance and compatibility with the $\mathcal{OU}$ operator.
\end{itemize}

The model \eqref{maineq2} has the {physical interpretation} of diffusion with Gaussian-weighted dispersion, while equation \eqref{maineq1} introduces a drift term analogous to confinement potentials. This difference leads to a mathematical trade-off: the drift simplifies the spectral structure of \eqref{maineq1}, allowing more direct Strichartz estimates (Theorem \ref{mainthm1}), whereas \eqref{maineq2} requires Keel–Tao’s machinery (Section \ref{Stri}).

Functionally, \eqref{maineq2} is more natural from the conservation law and variational perspectives, while \eqref{maineq1} is more amenable to linear analysis and dynamical arguments (e.g., small data scattering). The two are connected via a Gaussian gauge transformation:
\begin{align}
v := e^{-\frac{\alpha^2}{2}} u,
\end{align}
which maps \eqref{maineq1} into a ``modified" form resembling \eqref{maineq2} (see Introduction). Yet, this transformation introduces nontrivial weights into the nonlinearity and complicates higher-order analysis.

 The reason why we put the weight to the nonlinearity in \eqref{maineq1} is explained in the introduction (in view of the relations with \eqref{maineq2}) and in Section \ref{WP} (in view of the nonlinear estimates). It is fine to consider a general nonlinearity $F(u)$. Furthermore, it is also interesting to investigate what is the ``sharp" (minimal) weight requirement for establishing well-posedness for \eqref{maineq1}.

Overall, the dichotomy between \eqref{maineq2} and \eqref{maineq1} showcases two complementary approaches to incorporating partial $\mathcal{OU}$ operators in dispersive dynamics: one prioritizing physical derivation, and the other emphasizing spectral tractability. This duality suggests that future studies on ``NLS with noncompact confinement" may flexibly adopt either formulation depending on the analytic goals.

\subsection{Remarks on ``NLS on mixed geometry"}\label{sec:8.2}

The nonlinear Schr\"odinger equations considered in this paper—including those on waveguide manifolds $\mathbb{R}^m \times \mathbb{T}^n$, with partial harmonic confinement, and with $\mathcal{OU}$ drift—can all be viewed as examples of a broader class of models we call ``\textit{NLS on mixed geometry}.'' In such systems, free Euclidean directions coexist with confining directions, the latter realized through periodicity, external potentials, or drift mechanisms. This framework captures a rich interplay between dispersion and geometry, and provides a flexible setting for analyzing how confinement in one direction affects nonlinear dynamics in the others.

These confinement mechanisms differ in nature but share a common spectral feature: they induce localization through non-Euclidean structure. Periodic confinement, as in waveguides, arises from compactness of the torus $\mathbb{T}^n$ and leads to a discrete spectrum from $-\Delta_{\mathbb{T}^n}$. Harmonic potentials produce localization via quadratic growth at infinity, and their associated operator $-\Delta + |x|^2$ has an explicitly known discrete spectrum. The $\mathcal{OU}$ drift operator $\Delta_\alpha - \alpha \cdot \nabla_\alpha$, acting in a Gaussian-weighted space, generates a discrete spectrum as well, and can be interpreted as modeling a type of ``soft confinement'' through drift rather than geometry or potential.

\begin{table}[htbp]
\centering
\begin{tabular}{l|l|l}
\textbf{Model} & \textbf{Confinement Operator} & \textbf{Spectrum} \\ \hline
NLS on Waveguide Manifold & $-\Delta_{\mathbb{T}^n}$ & Discrete $\{k^2\}_{k\in\mathbb{Z}^n}$ \\
NLS with Partial Harmonic Potential & $-\Delta + |x|^2$ & Discrete $\{2|\beta|+n\}_{\beta\in\mathbb{N}^n}$ \\
NLW with $\mathcal{OU}$ Drift & $\Delta_\alpha - \alpha\cdot\nabla_\alpha$ & Discrete $\{n\}_{n\in\mathbb{N}}$ \\
\end{tabular}
\vspace{0.5cm}
\caption{Confinement mechanisms and spectral features}
\label{table1}
\end{table}

The structure of the confinement has a direct impact on the morphology of solutions. In waveguide settings, one observes modulated patterns akin to Bloch waves, shaped by the underlying periodicity. In the case of harmonic potentials, solutions decay like Gaussians and are often expanded in terms of Hermite functions. For the $\mathcal{OU}$ model, the solutions exhibit exponential decay in the confined direction, but with an inherent asymmetry induced by the drift term, distinguishing it from both the periodic and potential cases.

From a physical perspective, this unified viewpoint reveals several possibilities. The $\mathcal{OU}$ model, for instance, could describe waveguides with random or imperfect confinement, where traditional periodicity is replaced by a smoother stochastic effect. In Bose–Einstein condensates, drift-type confinement may provide a more realistic model for symmetric or asymmetric traps beyond the idealized quadratic potential. Furthermore, such frameworks may serve as toy models for quantum gases confined by soft or nonlocal mechanisms.

Several open problems arise naturally in this context. One intriguing direction is to explore stochastic generalizations of the $\mathcal{OU}$ operator, incorporating random noise or time-dependent drift, with potential connections to stochastic NLS. It is also unclear whether there exists a geometric realization of $\mathcal{OU}$-type drift through a Laplace–Beltrami operator on some noncompact Riemannian manifold. Another question is whether Morawetz-type inequalities can be extended or adapted to incorporate stochastic confinement and whether such mechanisms enhance or inhibit long-time dispersive behavior.

Viewed through this lens, the $\mathcal{OU}$-type NLS models offer a new perspective within the generalized waveguide paradigm. They form a distinct class of ``soft-confined dispersive systems,'' situated between deterministic and stochastic models. This suggests a more flexible and potentially unifying approach for analyzing dispersive PDEs in geometries where compactness is replaced by spectral localization.

We conclude this discussion by presenting a broader classification of confinement mechanisms that may appear in dispersive systems. These include not only the periodic, potential, and drift types mentioned earlier, but also more exotic scenarios such as fractional diffusion, magnetic fields, and stochastic potentials. Each introduces its own spectral structure and physical interpretation, enriching the taxonomy of NLS models on mixed or hybrid geometries.

\begin{table}[htbp]
\centering
\begin{tabular}{l|l|l|l}
\textbf{Confinement Type} & \textbf{Operator} & \textbf{Spectrum} & \textbf{Physical Interpretation} \\ \hline
Periodic Geometry& $-\Delta_{\mathbb{T}^n}$ & Discrete ($k^2$) & Optical fibers \\
Standard Potential & $-\Delta + V(x)$ & Depends on $V(x)$ & Bose-Einstein condensate \\
($\mathcal{OU}$) Drift & $\Delta - x\cdot\nabla$ & Discrete ($n\in\mathbb{N}$) & Brownian motion \\
Fractional Laplacian& $(-\Delta)^s$ ($s\in(0,1)$) & Continuous, heavy-tailed & Lévy flights \\
Magnetic-type& $(i\nabla - A)^2$ & Landau levels ($2n+1$) & Quantum Hall systems \\
Stochastic-type& $\mathcal{L}_\omega$ (random) & Dense pure point & Anderson localization \\
\end{tabular}
\vspace{0.5cm}
\caption{Extended classification of confinement mechanisms}
\end{table}

\subsection{Future directions}\label{sec:8.3}
Finally, in this subsection, we list some future directions on research line: ``NLS with $\mathcal{OU}$ operators.". As mentioned in the introduction, this research line is wide open, and many related problems can be further investigated.
\begin{itemize}
    \item \emph{The pure $\mathcal{OU}$ case.} In the current paper, we consider the ``partial $\mathcal{OU}$ case" as in \eqref{maineq2} and \eqref{maineq1}. In fact, one may also consider the ``pure $\mathcal{OU}$ case" (and study the well-posedness and even long time dynamics) in the following sense,
\begin{equation}
    (i\partial_t+(\Delta_{x}-x \cdot \nabla))u=F(u),\quad u(0,x)=u_0\in X,
\end{equation}
where $x \in \mathbb{R}^d$ ($d\geq 1$), $F(u)$ is the nonlinearity and $X$ is for a proper initial space. One may also consider the divergent analogue as \eqref{maineq2},
\begin{equation}
    i\partial_t u+\nabla_x \cdot (\nabla_x u e^{-\frac{\alpha^2}{2}})=F(u),\quad u(0,x)=u_0\in X.\footnote{The conservation laws for these two models can be deduced as in Section \ref{Cons}  with few modifications.}
\end{equation}

\item \emph{Other models with $\mathcal{OU}$ operators.} One may consider other NLS models with $\mathcal{OU}$ operators by changing the nonlinearity (including the sign), the whole dimension, the dimension for the $\mathcal{OU}$ direction, and even the initial space. Different model may behave differently and new ingredients (such as delicate fractional calculus in the weighted setting) may be required. One may also consider some other nonlinear dispersive equations (such as nonlinear wave equations) rather than NLS with (partial or pure) $\mathcal{OU}$ operators.

\item \emph{The focusing case.} One may investigate the focusing analogues of \eqref{maineq2} and \eqref{maineq1}. The Strichartz estimates and local theory work as in the defocusing case, while more ingredients are needed to deal with the long time problem. It is interesting to investigate large-data long-time dynamics in the focusing regime, which requires a deeper understanding of the threshold dynamics and ground state solutions under $\mathcal{OU}$ confinement. See Section \ref{foc} for blow-up results of \eqref{maineq2} in the focusing setting. As mentioned in Section \ref{Morawetz}, even for the defocusing case, the large data scattering in energy space is not clear.

\item \emph{The stochastic version.} A natural extension is to couple the $\mathcal{OU}$ operator with noise, leading to stochastic NLS with multiplicative Gaussian measure as follows
\begin{align}
i\partial_t u + \Delta_x u + \mathcal{OU}\, u = |u|^p u + \sigma u \circ dW_t,
\end{align}
  and it is natural to conjecture: ``Noise-induced regularization may suppress blow-up and ergodic invariant measures under drift confinement." For classical NLS, we refer to \cite{fan2023long} and the references therein for more details.
\end{itemize}

\bibliographystyle{abbrv}

\bibliography{OUNLS}

\end{document}